\documentclass[11pt]{amsart}

\usepackage{amssymb}
\usepackage{mathrsfs}
\usepackage{amsfonts}
\usepackage{latexsym,amsmath,amsthm,amssymb,amsfonts}
\usepackage[usenames]{color}
\usepackage{xspace,colortbl}
\usepackage{graphicx}
\usepackage{tipa}
\usepackage{amsthm}

\newcommand{\be}{\begin{equation}}
\newcommand{\ee}{\end{equation}}
\newcommand{\beq}{\begin{eqnarray}}
\newcommand{\eeq}{\end{eqnarray}}

\newtheorem{prop}{Proposition}[section]

\newtheorem{remark}[prop]{Remark}

\def\begeq{\begin{equation}}
\def\endeq{\end{equation}}

\def\odot{\setbox0=\hbox{$\bigcirc$}\relax \mathbin {\hbox
to0pt{\raise.5pt\hbox to\wd0{\hfil $\wedge$\hfil}\hss}\box0 }}

\numberwithin{equation} {section}

\numberwithin{equation}{section}
\newtheorem{theorem}{\bf Theorem}[section]
\newtheorem{proposition}[theorem]{\bf Proposition}
\newtheorem{definition}[theorem]{\bf Definition}
\newtheorem{lemma}[theorem]{\bf Lemma}

\newtheorem{corollary}[theorem]{\bf Corollary}

\allowdisplaybreaks

\begin{document}
\title[Pogorelov type estimate for sum Hessian quotient equations]
 {Pogorelov type interior $C^2$ estimate for sum Hessian quotient equations on Riemannian manifolds}\thanks{\it {The research is partially supported by NSFC (No. 12261105).}}


\author{
Guanghan Li,~~Chenyang Liu*}
\address{
School of Mathematics and Statistics, Wuhan University, Wuhan 430072, China
}

\email{ghli@whu.edu.cn,liuchy@whu.edu.cn}

\thanks{$\ast$ Corresponding author}

\date{}
\maketitle
\begin{abstract}
In this paper, we mainly study the interior $C^{2}$ estimates for a class of sum Hessian quotient equations on Riemannian manifolds. For $0\leq l<k< n$, we establish interior $C^{2}$ estimates at the center of a geodesic ball. Let $\Omega$ be a bounded domain (with smooth boundary) on Riemannian manifold. For $0\leq l<k\leq n$, we establish Pogorelov type estimates on $\Omega$ with the vanishing Dirichlet boundary condition.
\end{abstract}

\section{Introduction} \label{S1}

Pogorelov type interior $C^{2}$ estimates play a key role in the theory of fully nonlinear elliptic equations. In 1959, Heinz~\cite{Hein59} proved a pure interior $C^2$ estimate for the two-dimensional Monge-Amp\`{e}re 
equation $\det(\nabla^2u)=f$. When $n\ge3$, Pogorelov~\cite{Pogo78} showed that pure interior estimates no longer hold. He established an interior $C^2$ estimate for all dimensions with Dirichlet boundary conditions, which is now called Pogorelov type interior $C^{2}$ estimate, or simply Pogorelov type estimate. Later, Chou-Wang \cite{CW01, Wang09} extended this result to $k$-Hessian equations. Explicitly, for a bounded domain $\Omega\subset \mathbb{R}^{n}$ with smooth boundary $\partial\Omega$, they consider the following Dirichlet problem of the $k$-Hessian equation
\begin{equation}\label{k-hessian-1}
\left\{
\begin{aligned}
&\sigma_{k}(\nabla^{2}u)=f(x,u),\qquad&&~\mathrm{in}~\Omega,\\
&u=0, \qquad&&~\mathrm{on}~\partial\Omega,\\
\end{aligned}
\right.
\end{equation}
where $\sigma_{k}(\nabla^{2}u)$ denotes $\sigma_{k}(\lambda(\nabla^{2}u))$ (see \eqref{sigma-1}), with $\lambda(\nabla^{2}u)$ being the eigenvalues of the Hessian matrix $\nabla^{2}u$. Chou-Wang proved that, in the case $f\geq 0$ and $u<0$, there exists a constant $\beta>0$ such that 
\begin{equation*}
\sup(-u)^{\beta} |\nabla^{2} u| \leq C
\end{equation*}
for any $k$-convex solution $u\in C^{2}(\Omega)$ to equation \eqref{k-hessian-1}. Here, $u$ is called $k$-convex if $\lambda(\nabla^{2}u)$ belongs to the Garding's cone $\Gamma_{k}$ (see \eqref{Gamma-1}). Subsequently, Li-Ren-Wang \cite{LRW16} established Pogorelov type estimates for $(k+1)$-convex solutions to the Dirichlet problem of equations \eqref{k-hessian-1}, where $f$ further depends on the gradient $\nabla u$.

In \cite{CH21},  Chu-Jiao studied the prescribed curvature problem for hypersurfaces $M^{n}$ in $\mathbb{R}^{n+1}$ and proved Pogorelov type estimates for the Dirichlet problem of the corresponding $k$-Hessian equation:
\begin{equation}\label{k-hessian-2}
\left\{
\begin{aligned}
&\sigma_{k}(\tilde{\eta})=f(x,u,\nabla u),\qquad&&~\mathrm{in}~\Omega,\\
&u=0, \qquad&&~\mathrm{on}~\partial\Omega,\\
\end{aligned}
\right.
\end{equation}
where $\tilde{\eta} = (\tilde{\eta}_1,\dots,\tilde{\eta}_n)$ are the eigenvalues of $(\Delta u)I - \nabla^{2}u$, and $\Omega\subset \mathbb{R}^{n}$ is a bounded domain with smooth boundary $\partial\Omega$. The operator $\sigma_{k}$ applied to the Hessian matrix $(\Delta u)I - \nabla^{2}u$ first came from the Gauduchon conjecture \cite{Gau84, STW17} in complex geometry. Harvey-Lawson \cite{HL11,HL12} introduced the concept that if the real Hessian matrix is replaced by the complex Hessian matrix, a function $u$ is called $(n-1)$-plurisubharmonic provided $(\Delta u)I - \nabla^{2}u$ is nonnegative definite. Thus, complex Monge-Amp\`{e}re type equations for $(n-1)$-plurisubharmonic functions can be defined as follows
\begin{equation}\label{k-hessian-3}
\det\left(\left(\sum_{m=1}^{n}\frac{\partial^{2}u}{\partial z_{m}\partial \bar{z}_{m}}\right)\delta_{ij}-\frac{\partial^{2}u}{\partial z_{i}\partial \bar{z}_{j}}\right)=f.
\end{equation}
For strict pseudo-convex domains in the complex Euclidean $n$-space $\mathbb{C}^{n}$, the Dirichlet problem of \eqref{k-hessian-3} with positive $f$ was solved by Li \cite{L04}. Tosatti-Weinkove \cite{TW17} studied equation \eqref{k-hessian-3} on K$\ddot{\mathrm{a}}$hler manifolds. For more references, we refer readers to \cite{Dong23,FWW10,S18} and references therein.

Recently, Ren-Wang \cite{RW25-2} considered the following equation, which is called the sum Hessian quotient equation:
\begin{equation}\label{Hessian quo-1}
\frac{\sigma_{k}(\tilde{\eta})+\alpha\sigma_{k-1}(\tilde{\eta})}{\sigma_{l}(\tilde{\eta})+\alpha\sigma_{l-1}(\tilde{\eta})}=f(x,u,\nabla u),
\end{equation}
where $\tilde{\eta} = (\tilde{\eta}_1,\dots,\tilde{\eta}_n)$ are the eigenvalues of $(\Delta u)I - \nabla^{2}u$, $f$ is a positive function, $\alpha\geq0$ and $0\leq l<k\leq n$. For $0 \leq l < k < n$, they proved interior estimates for \eqref{Hessian quo-1} and Pogorelov type estimates for its Dirichlet problem on bounded domains in $\mathbb{R}^{n}$. In the case $k=n$ they obtained weak Pogorelov type estimates for $0 \leq l < n-1$. Ren-Wang \cite{RW25-1} obtained similar results for the sum Hessian equation, i.e., the case $l=0$. Pogorelov type interior $C^{2}$ estimates play an essential role in the regularity theory, which is a central part of the prescribed curvature problem. In \cite{CTX23}, Chen-Tu-Xiang considered the prescribed curvature problem for Hessian quotient equations with Dirichlet conditions on Riemannian manifolds and proved the existence of solutions. Li and Sheng \cite{LS13} studied the prescribed Weingarten curvature problem for closed hypersurfaces on Riemannian manifolds. A natural question is whether the corresponding Pogorelov type interior $C^{2}$ estimates can be established on Riemannian manifolds.

Let $M^n$ be a complete $n$-dimensional Riemannian manifold with Riemannian metric $g$. In this paper, we consider the interior $C^{2}$ estimates for the following sum Hessian quotient equation:
\begin{equation}\label{Hessian quo-2}
\frac{\sigma_{k}(\eta)+\alpha\sigma_{k-1}(\eta)}{\sigma_{l}(\eta)+\alpha\sigma_{l-1}(\eta)}=f(x,u,\nabla u),
\end{equation}
where $u(x)$ is a function defined over $\Omega\subset M^n$, $f$ is a positive function, $\alpha\geq 0$, $0\leq l<k\leq n$, $\eta= (\eta_1,\dots,\eta_n)$ are the eigenvalues of $(\tau\Delta u)I - \nabla^{2}u$ with $\tau\geq1$, and $\nabla$, $\Delta$, $\nabla^2$ denote the gradient, Laplace, Hessian operators on $M^{n}$, respectively. 
Different from the Euclidean case, on a general Riemannian manifold interchanging covariant derivatives gives rise to extra terms involving the curvature tensor, which makes it more difficult to obtain interior estimates for equation \eqref{Hessian quo-2}. Moreover, another important ingredient in the Euclidean proof is the fact that the second derivatives of the Euclidean distance function are simply constant (in fact, $\nabla^{2}(|x|^{2})=2I$), which allows it to be used in constructing test functions to absorb bad terms in the maximum principle argument. To the best of our knowledge, on a general Riemannian manifold no such simple distance function is available. To overcome these difficulties, we develop in this paper a systematic approach. By applying the Ricci identity, we give the explicit form of the curvature terms arising from the interchange of covariant derivatives, and prove that these terms can all be controlled. Then we define the distance function on the Riemannian manifold and apply the Hessian comparison theorem to control the Hessian of its square, which is then used to absorb the bad terms, serving the same role as $|x|^{2}$ in the Euclidean case. And we resolve the difficulties arising from the singularity of the distance function at the origin and the inapplicability of the comparison theorem on the cut locus, which significantly reduces the limitations on the domain where the estimates are established.

We first give the definitions of $k$-convex and $k$-admissible.

\begin{definition}
Let $\Omega$ be a domain and $u\in C^{2}(\Omega)$. Denote by $\lambda(x)=(\lambda_{1}(x),\dots,\lambda_{n}(x))$ the eigenvalues of the Hessian $\nabla^{2}u(x)$, a function $u$ is called $k$-convex if $\lambda(x)$ are in $\Gamma_{k}$ for all $x\in\Omega$, where $\Gamma_{k}$ is the Garding's cone
\begin{equation}\label{Gamma-1}
\Gamma_{k}=\{\lambda\in\mathbb{R}^{n}|\sigma_{m}(\lambda)>0, m=1,\dots, k\}.
\end{equation}
Similarly, a function $u$ is called $k$-admissible if $\lambda(x)$ are in $\Gamma'_{k}$ for all $x\in\Omega$, where
\begin{equation}\label{Gamma-2}
\Gamma'_{k}=\{\lambda\in\mathbb{R}^{n}|\sigma_{m}(\eta)>0, m=1,\dots, k\}.
\end{equation}
\end{definition}

To state our main result, we next introduce some notation. For any $x\in M^{n}$, denote by $\operatorname{Cut}(x)$ the cut locus of $x$ in $M^{n}$. Let $\Omega$ be a bounded domain in $M^{n}$ with smooth boundary $\partial\Omega$.
For any $x \in M^n$, let $p_0$ be a fixed point. The geodesic distance between $x$ and $p_0$ is defined by
\begin{equation*}
\rho(x) = \operatorname{dist}_{M^{n}}(x, p_{0}).
\end{equation*}
If there exists a point $p_0 \in \Omega$ such that $\overline{\Omega} \subset B_r(p_0)$, then we say that $\Omega$ has \emph{covering radius} at most $r$, and we denote this by
\begin{equation*}
\operatorname{rad}(\Omega) \leq r.
\end{equation*}
Let $\operatorname{Sec}(\Pi)$ denote the sectional curvature of the plane $\Pi$ in $M^{n}$, define
\begin{equation*}
\overline{K}_{\Omega} := \sup_{x \in \Omega,~\Pi \subset T_{x} M^n} \operatorname{Sec}(\Pi),\quad\quad \underline{K}_{\Omega} := \inf_{x \in \Omega,~\Pi \subset T_{x} M^n} \operatorname{Sec}(\Pi), 
\end{equation*}
where $\Pi$ runs over all $2$-dimensional subspaces of $T_{x} M^n$.
Riemannian manifolds $M^n$ of constant sectional curvature are precisely the space forms, i.e., $\operatorname{Sec}(M^n)\equiv c$. 
If $\operatorname{Sec}(M^n)\equiv 0$, then $M^n$ is the Euclidean space $\mathbb{R}^n$; if $\operatorname{Sec}(M^n)\equiv -1$, then $M^n$ is the hyperbolic space $\mathbb{H}^n$; if $\operatorname{Sec}(M^n)\equiv 1$, then $M^n$ is the sphere $\mathbb{S}^n$. 
Here we list the main results.

\begin{theorem}\label{maintheorem1}
For any point $o\in M^{n}$, suppose that $0\leq l<k<n$, $u\in C^{4}(B_{r}(o))$ is a $(k-1)$-admissible solution to the sum Hessian quotient equation \eqref{Hessian quo-2}, where $B_{r}(o)$ is a geodesic ball of radius $r$ centered at the origin in $M^{n}$, $f\in C^{2}(B_{r}(o)\times\mathbb{R}\times\mathbb{R}^{n})$ with $0<m_{1}\leq f\leq m_{2}$. Then
\begin{equation}\label{estimates-1}
|\nabla^{2}u(o)|\leq C(1+\sup|\nabla u|),
\end{equation}
where $C$ is a positive constant depending only on $n$, $k$, $l$, $r$, $m_{1}$, $m_{2}$, $||R_{ijkl}||_{C^{1}}$ and $||f||_{C^{2}}$.
\end{theorem}

Following the proof of Theorem \ref{maintheorem1}, let $\Omega$ be a bounded domain with smooth boundary $\partial\Omega$ in Riemannian manifold $M^{n}$. Then we derive the Pogorelov type estimates for the Dirichlet problem of equation \eqref{Hessian quo-2}:
\begin{equation}\label{Hessian quo-3}
\left\{
\begin{aligned}
&\frac{\sigma_{k}(\eta)+\alpha\sigma_{k-1}(\eta)}{\sigma_{l}(\eta)+\alpha\sigma_{l-1}(\eta)}=f(x,u,\nabla u),\qquad&&~\mathrm{in}~\Omega,\\
&u=0, \qquad&&~\mathrm{on}~ \partial\Omega.\\
\end{aligned}
\right.
\end{equation}

\begin{theorem}\label{maintheorem2}
Suppose that $0\leq l<k<n$, $u\in C^{4}(\Omega)\cap C^{2}(\overline{\Omega})$ is a $(k-1)$-admissible solution to the Dirichlet problem \eqref{Hessian quo-3} in a bounded domain $\Omega\subset M^{n}$, $f\in C^{2}(\overline{\Omega}\times\mathbb{R}\times\mathbb{R}^{n})$ with $f>0$. Then
\begin{equation*}
(-u)|\nabla^{2}u|\leq C,
\end{equation*}
where $C$ is a positive constant depending only on $n$, $k$, $l$, $||f||_{C^{2}}$, $||R_{ijkl}||_{C^{1}}$ and $||u||_{C^{1}}$.
\end{theorem}

For the case $k=n$, we have the following result.

\begin{theorem}\label{maintheorem3}
Suppose that $0\leq l<n-1$ and $k=n$, $u\in C^{4}(\Omega)\cap C^{2}(\overline{\Omega})$ is a $(n-1)$-admissible solution to the Dirichlet problem \eqref{Hessian quo-3} in a bounded domain $\Omega\subset M^n$ (when $\overline{K}_{\Omega}>0$, we additionally require $\operatorname{rad}(\Omega) <\pi/2\sqrt{\overline{K}_{\Omega}}$), $f\in C^{2}(\overline{\Omega}\times\mathbb{R}\times\mathbb{R}^{n})$ with $f>0$. Then
\begin{equation*}
(-u)|\nabla^{2}u|\leq C,
\end{equation*}
where $C$ is a positive constant depending only on $n$, $l$, $\tau$, $||f||_{C^{2}}$, $||R_{ijkl}||_{C^{1}}$ and $||u||_{C^{1}}$.
\end{theorem}

Note that there are many results on Pogorelov type estimates for other operators on space forms. In this paragraph, We mention some results that correspond to \eqref{Hessian quo-3}. 
In the Euclidean case (i.e., $\operatorname{Sec}(M^{n})\equiv 0$), for $\tau=1$, Ren-Wang \cite{RW25-2} obtained the same Pogorelov type estimates as in Theorem \ref{maintheorem2} and Theorem \ref{maintheorem3}. If $l=0$ and $\tau=1$, Ren-Wang \cite{RW25-1} established Pogorelov type estimates for $0<k<n$, and weaker Pogorelov type estimates for $k=n$. Replacing the eigenvalue matrix in \eqref{Hessian quo-3} with $D^{2}u$, Liu-Ren \cite{LR23} established Pogorelov type estimates for the case $l=0$. If $\tau=1$, $\alpha=0$, and $l=0$, Chu-Jiao \cite{CH21} established Pogorelov type estimates for $k$-admissible solutions to $\sigma_k=f$. If $\tau=1$ and $\alpha=0$, Chen-Tu-Xiang \cite{CTX21} obtained Pogorelov type estimates for $k$-admissible solutions to $\sigma_k/\sigma_l=f$, where $l+2\leq k\leq n$. When $\operatorname{Sec}(M^{n})\equiv -1$ and $\alpha=0$, Liu-Mao-Zhao \cite{LMZ22} obtained the same Pogorelov type estimates as Chen-Tu-Xiang for $k$-admissible solutions to $\sigma_k/\sigma_l=f$. The results in \cite{CTX21,CH21,LMZ22,RW25-1,RW25-2} are covered by Theorem \ref{maintheorem2} and Theorem \ref{maintheorem3}.

\begin{remark}
\normalfont\upshape (1) By Theorem \ref{maintheorem2} and Theorem \ref{maintheorem3}, if $C^0$ and $C^1$ estimates can be established for $(k-1)$-admissible solutions of \eqref{Hessian quo-3}, then interior $C^2$ estimates follow immediately.

(2) If $M^n$ can be embedded into a $(n+1)$-dimensional manifold $N^{n+1}$, it is precisely a hypersurface therein. For instance, $\mathbb{H}^n$ can be embedded into the Lorentz-Minkowski space $\mathbb{R}^{n+1}_{1}$, while $\mathbb{R}^n$ and $\mathbb{S}^n$ can be embedded into the Euclidean space $\mathbb{R}^{n+1}$. If we take $u(x)$ as the graph function of a hypersurface on the bounded domain $\Omega\subset M^n$, then the existence of $u(x)$ is equivalent to the existence of a hypersurface in $N^{n+1}$, while the Dirichlet problem \eqref{Hessian quo-3} corresponds to a prescribed curvature problem. For many operators, once a priori estimates are available, the existence of solutions can be obtained via the continuity method or degree theory. This yields the existence of corresponding hypersurfaces in Riemannian manifolds. Since $C^0$ and $C^1$ estimates are relatively easy to obtain in PDE theory, the $C^2$ estimate is the most challenging part. This is our motivation for considering Pogorelov estimates on Riemannian manifolds.

(3) According to the Cartan–Hadamard theorem, if $M^{n}$ is a complete, simply connected Riemannian manifold with nonpositive sectional curvature, then $\operatorname{Cut}(p)=\emptyset$ for all $p\in M^{n}$ (see Corollary~\ref{cut-locus-empty}). Hence, in this case we do not need to consider the influence of the cut locus, and the proof simplifies.

(4) In the proof, one may observe that we use the boundedness of the sectional curvature to apply the comparison theorem. However, for any bounded closed domain $\overline{\Omega}$ in a Riemannian manifold, since $\overline{\Omega}$ is compact and the curvature tensor is continuous, there exists a positive constant $K$ such that $\overline{K}_{\Omega}\leq K$ and $\underline{K}_{\Omega}\geq -K$. This property is automatic from the Riemannian manifold itself, so we do not impose it as an additional assumption.
Then Lemma~\ref{Hess-compar} shows that it suffices to bound the sectional curvature only on that domain and along the geodesics joining it to the chosen base point $p_{0}$.
\end{remark}

The paper is organized as follows. In Section \ref{S2}, some useful formulas and basic properties will be listed. In Section \ref{S3}, we establish interior estimates for the sum Hessian quotient equation for $0\leq l<k<n$. The proofs of Theorem \ref{maintheorem2} and Theorem \ref{maintheorem3} are given respectively in Section \ref{S4} and Section \ref{S5}.

\section{Preliminaries} \label{S2}

In this section, we present some preliminaries on Riemannian manifolds and the properties of elementary symmetric polynomials.

Let $g$ be the Riemannian metric on $M^{n}$, and $\nabla$, $\Delta$, $\nabla^{2}$ denote the gradient, Laplace, and Hessian operators on $M^{n}$ induced by $g$, respectively. For a $(s, r)$-tensor field $\mathscr{T}$ on $M^{n}$, its covariant derivative $\nabla\mathscr{T}$ is a $(s, r+1)$-tensor field given by
\begin{equation*}
\begin{aligned}
&\nabla\mathscr{T}(Y^{1},\dots,Y^{s},X_{1},\dots,X_{r},X)\\
&=~\nabla_{X}\mathscr{T}(Y^{1},\dots,Y^{s},X_{1},\dots,X_{r})\\
&=~X(\mathscr{T}(Y^{1},\dots,Y^{s},X_{1},\dots,X_{r}))-\mathscr{T}(\nabla_{X}Y^{1},\dots,Y^{s},X_{1},\dots,X_{r})\\
&\quad\quad-\cdots-\mathscr{T}(Y^{1},\dots,Y^{s},X_{1},\dots,\nabla_{X}X_{r}),
\end{aligned}
\end{equation*}
where $Y^{1},\dots,Y^{s}$ are smooth $1$-forms, and $X_{1},\dots,X_{r},X$ are smooth tangent vector fields. Its components in local coordinates are denoted by
\begin{equation*}
\mathscr{T}_{k_{1}\cdots k_{r}, k_{r+1}}^{l_{1}\cdots l_{s}},
\end{equation*}
where $1\leq l_{i}, k_{j}\leq n$ with $i=1, 2, \cdots, s$ and $j=1, 2, \cdots, r+1$. In the following, a comma “,” in subscript of a given tensor means doing covariant derivatives. For simplicity, we will omit it when no confusion arises. The second covariant derivative of $\mathscr{T}$ is defined analogously:
\begin{equation*}
\nabla^{2}\mathscr{T}(Y^{1},\dots,Y^{s},X_{1},\dots,X_{r},X,Y)=(\nabla_{Y}(\nabla\mathscr{T}))(Y^{1},\dots,Y^{s},X_{1},\dots,X_{r},X)
\end{equation*}
and then its components in local coordinates are denoted by
\begin{equation*}
\mathscr{T}_{k_{1}\cdots k_{r}, k_{r+1}k_{r+2}}^{l_{1}\cdots l_{s}},
\end{equation*}
where $1\leq l_{i}, k_{j}\leq n$ with $i=1, 2, \cdots, s$ and $j=1, 2, \cdots, r+2$. Similarly, the higher order covariant derivatives of $\mathscr{T}$ are given as follows
\begin{equation*}
\nabla^{3}\mathscr{T}=\nabla(\nabla^{2}\mathscr{T}),\quad \nabla^{4}\mathscr{T}=\nabla(\nabla^{3}\mathscr{T}),\cdots,
\end{equation*}
and so on. 

For any tangent vector fields $X, Y, Z$ on $M^{n}$, the Riemannian curvature $(1, 3)$-tensor $R$ induced by the metric $g$ is defined as
\begin{equation*}
R(X, Y)Z=-\nabla_{X}\nabla_{Y}Z+\nabla_{Y}\nabla_{X}Z+\nabla_{[X, Y]}Z,
\end{equation*}
where $[\cdot,\cdot]$ denotes the Lie bracket. In a local coordinate chart $\{\xi^{i}\}_{i=1}^{n}$ of $M^{n}$, the component of the curvature tensor $R$ is defined by
\begin{equation*}
R\left(\frac{\partial}{\partial \xi^{i}},\frac{\partial}{\partial \xi^{j}}\right)\frac{\partial}{\partial \xi^{k}}=R^{h}_{kij}\frac{\partial}{\partial \xi^{h}},
\end{equation*}
and
\begin{equation*}
g\left(R\left(\frac{\partial}{\partial \xi^{k}},\frac{\partial}{\partial \xi^{l}}\right)\frac{\partial}{\partial \xi^{i}}, \frac{\partial}{\partial \xi^{j}}\right)=R_{ikl}^{h}g_{hj}=R_{ijkl},
\end{equation*}
where $g_{ij}:=g\left(\frac{\partial}{\partial \xi^{i}}, \frac{\partial}{\partial \xi^{j}}\right)$. Then, we give the well-known Ricci identity
\begin{equation}\label{ricci-ident}
\mathscr{T}_{j_{1}\cdots j_{s}, kl}^{i_{1}\cdots i_{r}}-\mathscr{T}_{j_{1}\cdots j_{s}, lk}^{i_{1}\cdots i_{r}}=-\sum\limits_{a=1}^{r}\mathscr{T}_{j_{1}\cdots j_{s}}^{i_{1}\cdots i_{a-1}ii_{a+1}\cdots i_{r}}R^{i_{a}}_{ikl}
+\sum\limits_{b=1}^{s}\mathscr{T}_{j_{1}\cdots j_{b-1}jj_{b+1}\cdots j_{s}}^{i_{1}\cdots i_{r}}R^{j}_{j_{b}kl}.
\end{equation}
If $u$ is a sufficiently smooth function on $M^{n}$, its second covariant derivative is 
\begin{equation*}
\nabla_{i}\nabla_{j}u=\frac{\partial^{2}u}{\partial\xi_{i}\partial\xi_{j}}-\Gamma_{ij}^{k}\frac{\partial u}{\partial\xi_{k}},
\end{equation*}
where $\Gamma_{ij}^{k}$ are the Christoffel symbols, $\partial/\partial \xi_{i}$ denotes the ordinary partial derivative. Throughout this paper, we write $u_{ij}$ for $\nabla_i\nabla_j u$, and similarly we write $u_{ijk}$ and $u_{ijkl}$ for higher order covariant derivatives. By the Ricci identity \eqref{ricci-ident}, one derives
\begin{equation}\label{3-eq7-1}
u_{i11}=u_{11i}+R^{m}_{1i1}u_{m}.
\end{equation}
\begin{equation}\label{3-eq7-2}
u_{ii11}=u_{11ii}+2R_{ii1}^{1}u_{11}+2R_{1i1}^{i}u_{ii}+u_{j}(R_{ii1}^{j})_{,\xi_{1}}+u_{j}(R_{1i1}^{j})_{,\xi_{i}}.
\end{equation}

According to \cite{CL02}, we state the following Cartan-Hadamard Theorem.

\begin{lemma}[Cartan--Hadamard]\label{cartan-had}
Let $M^{n}$ be a complete, simply connected Riemannian manifold with sectional curvature $K \leq 0$ everywhere.
Then for any point $p \in M^{n}$, the exponential map 
\begin{equation*}
\exp_{p} \colon T_{p}M^{n} \to M^{n}
\end{equation*}
is a global diffeomorphism. Consequently, any two points in $M^{n}$ are joined by a unique minimizing geodesic, and the distance function $\rho(x) = \operatorname{dist}(p,x)$ is smooth on $M^{n} \setminus \{p\}$.
\end{lemma}

\begin{corollary}\label{cut-locus-empty}
Under the assumptions of Lemma~\ref{cartan-had}, the cut locus of every point is empty, i.e.\ $\operatorname{Cut}(p) = \emptyset$ for all $p \in M^{n}$.
\end{corollary}

Under the same notation as before, for any point $x\in M^{n}\setminus \operatorname{Cut}(p_{0})\cup\{p_{0}\}$, we have the well-known Hessian Comparison Theorem.

\begin{theorem}[Hessian Comparison]\label{Hess-compar}
Assume that along every radial geodesic from $p_{0}$, the sectional curvatures of planes containing $\nabla\rho$ satisfy
\begin{equation*}
c_{1} \leq \operatorname{Sec}(\nabla\rho,\cdot) \leq c_{2},
\end{equation*}
where $c_1, c_2\in\mathbb{R}$ are constants, $\operatorname{Sec}(\nabla\rho,\cdot)$ denotes the sectional curvature of the 2-plane spanned by $\nabla\rho$ and an arbitrary tangent vector orthogonal to $\nabla\rho$, and if $c_2>0$, we additionally require $\rho(x)<\pi/\sqrt{c_2}$.
Define
\begin{equation*}
\mu_{c(t)}=
\begin{cases}
\sqrt{c} \cot(\sqrt{c} t), & c>0,\\[4pt]
\dfrac{1}{t}, & c=0,\\[8pt]
\sqrt{|c|} \coth(\sqrt{|c|} t), & c<0.
\end{cases}
\end{equation*}
Then, at all points where $\rho(x)$ is smooth, we have
\begin{equation*}
\mu_{c_{2}}(\rho)(g-d\rho\otimes d\rho) \leq \nabla^2\rho \leq \mu_{c_{1}}(\rho)(g-d\rho\otimes d\rho).
\end{equation*}
\end{theorem}

Next, some algebraic identities and properties of $\sigma_k$ and $S_k$ are recalled. These can be found in \cite{CTX21, LR23, RW25-1, RW25-2}, for convenience we list them here and provide proofs of some important ones. 

Let $\lambda=(\lambda_{1},\dots,\lambda_{n})\in\mathbb{R}^{n}$, recall the definition of elementary symmetric function for $1\leq k\leq n$,
\begin{equation}\label{sigma-1}
\sigma_{k}(\lambda)=\sum\limits_{1\leq i_{1}<\dots<i_{k}\leq n}\lambda_{i_{1}}\dots\lambda_{i_{k}}.
\end{equation}
We adopt the convention that $\sigma_{0}=1$ and $\sigma_{k}=0$ for $k>n$ or $k<0$.

As in \cite{RW25-1, RW25-2}, we define the sum Hessian function as
\begin{equation*}
S_{k}(\lambda)=\sigma_{k}(\lambda)+\alpha\sigma_{k-1}(\lambda).
\end{equation*}
for $\lambda=(\lambda_{1},\dots,\lambda_{n})\in\mathbb{R}^{n}$ and $1\leq k\leq n$. The notation $(\lambda|i)=(\lambda_1,\dots,\lambda_{i-1},\lambda_{i+1},\dots,\lambda_n)$ will be used throughout. From \cite[Proposition 2.1]{CTX21}, one has the following facts:

\begin{proposition}\label{prop-1}
Let $\lambda=(\lambda_{1},\dots,\lambda_{n})\in\mathbb{R}^{n}$ and $1\leq k\leq n$, then we have\\
(1)~$\Gamma_{1}\supset\Gamma_{2}\supset\dots\supset\Gamma_{n};$\\
(2)~$\sigma_{k-1}(\lambda|i)>0$ for $\lambda\in\Gamma_{k}$ and $1\leq i\leq n$;\\
(3)~$\sigma_{k}(\lambda)=\sigma_{k}(\lambda|i)+\lambda_{i}\sigma_{k-1}(\lambda|i)$ for $1\leq i\leq n$;\\
(4)~If $\lambda_{1}\geq\lambda_{2}\geq\dots\geq\lambda_{n},$ then $\sigma_{k-1}(\lambda|1)\leq\sigma_{k-1}(\lambda|2)\leq\dots\leq\sigma_{k-1}(\lambda|n)$ for $\lambda\in\Gamma_{k}$;\\
(5)~$\sum\limits_{i=1}^{n}\sigma_{k-1}(\lambda|i)=(n-k+1)\sigma_{k-1}(\lambda)$.
\end{proposition}

The following Newton-MacLaurin inequality will be used as well (see, e.g., \cite{LT94, T90}).

\begin{lemma}\label{N-M}
Let $\lambda\in\mathbb{R}^{n}$. For $0\leq l< k\leq n$, $r>s\geq 0$, $k\geq r$, $l\geq s$, we have
\begin{equation}\label{N-M-ieq}
\left[\frac{\sigma_{k}(\lambda)/C^{k}_{n}}{\sigma_{l}(\lambda)/C^{l}_{n}}\right]^{\frac{1}{k-l}}\leq\left[\frac{\sigma_{r}(\lambda)/C^{r}_{n}}{\sigma_{s}(\lambda)/C^{s}_{n}}\right]^{\frac{1}{r-s}},
\quad for~\lambda\in\Gamma_{k}.
\end{equation}
\end{lemma}

According to Li–Ren–Wang \cite{LRW19}, the sum Hessian operator $S_{k}(\lambda)$ has the admissible solution set:
\begin{equation*}
\widetilde{\Gamma}_{k}=\Gamma_{k-1}\cap \{\lambda|S_{k}>0\}.
\end{equation*}
Moreover, both $S_{k}^{1/k}(\lambda)$ and $[\frac{S_{k}(\lambda)}{S_{l}(\lambda)}]^{1/k-l}$ are concave in the set $\widetilde{\Gamma}_{k}$. Some useful properties are listed below (see, e.g., \cite{RW25-1, RW25-2}).

\begin{proposition}\label{prop-2}
Let $\lambda=(\lambda_{1},\dots,\lambda_{n})\in\mathbb{R}^{n}$ and $1\leq k\leq n$, then we have\\
(1)~$S^{ii}_{k}(\lambda):=\frac{\partial S_{k}(\lambda)}{\partial\lambda_{i}}=\sigma_{k-1}(\lambda|i)+\alpha\sigma_{k-2}(\lambda|i),\quad i=1,2,\dots,n;$\\
(2)~$S^{ii,jj}_{k}(\lambda):=\frac{\partial^{2} S_{k}(\lambda)}{\partial\lambda_{i}\partial\lambda_{j}}=S_{k-2}(\lambda|ij),\quad i,j=1,2,\dots,n;$\\
(3)~$S_{k}(\lambda)=\lambda_{i}S_{k-1}(\lambda|i)+S_{k}(\lambda|i),\quad i=1,2,\dots,n;$\\
(4)~$\sum\limits_{i=1}^{n}S_{k}(\lambda|i)=(n-k)S_{k}(\lambda)+\alpha\sigma_{k-1}(\lambda);$\\
(5)~$\sum\limits_{i=1}^{n}\lambda_{i}S_{k-1}(\lambda|i)=kS_{k}(\lambda)-\alpha\sigma_{k-1}(\lambda)$.
\end{proposition}

\begin{lemma}\label{prop-3}
We have:\\
(1)~$\widetilde{\Gamma}_{k}$ are convex cones, and $\widetilde{\Gamma}_{1}\supset\widetilde{\Gamma}_{2}\supset\dots\supset\widetilde{\Gamma}_{n};$\\
(2)~If $\lambda=(\lambda_{1},\dots,\lambda_{n})\in\widetilde{\Gamma}_{k}$ and $\lambda_{1}\geq\lambda_{2}\geq\dots\geq\lambda_{n}$, then $\lambda_{k-1}>0$ for $2\leq k\leq n$,
\begin{equation*}
S_{k-1}(\lambda|n)\geq S_{k-1}(\lambda|n-1)\geq\dots\geq S_{k-1}(\lambda|1)>0,
\end{equation*}
and
\begin{equation*}
S_{k-1}(\lambda|k)\geq c(n,k)S_{k-1}(\lambda),
\end{equation*}
where $c(n, k)$ is a positive constant only depending on $n$ and $k$, $1\leq k\leq n$;\\
(3)~If $\lambda=(\lambda_{1},\dots,\lambda_{n})\in\widetilde{\Gamma}_{k}$, then for any $(\xi_{1}, \dots, \xi_{n})$, we have
\begin{equation*}
\sum\limits_{i, j=1}^{n}\frac{\partial^{2}[\frac{S_{k}(\lambda)}{S_{l}(\lambda)}]}{\partial\lambda_{i}\partial\lambda_{j}}\xi_{i}\xi_{j}
\leq(1-\frac{1}{k-l})\frac{\Big[\sum_{i}\frac{\partial[\frac{S_{k}(\lambda)}{S_{l}(\lambda)}]}{\partial\lambda_{i}}\xi_{i}\Big]^{2}}{\frac{S_{k}(\lambda)}{S_{l}(\lambda)}}.
\end{equation*}
\end{lemma}
\begin{proof}
Statement (1) follows directly from the definition of $S_k$. The proof of (2) can be found in \cite{R24}, and that of (3) is given in \cite{RW25-2}.
\end{proof}

For $\lambda=(\lambda_{1}, \dots, \lambda_{n})$, let $\eta=(\eta_1, \dots, \eta_n)$ be defined by
\begin{equation*}
\eta_{i}=\tau\sum_{j=1}^{n}\lambda_{j}-\lambda_{i},
\end{equation*}
where $\tau\geq1$. When $\tau=1$, this reduces to the setting used in \cite{RW25-2} (i.e.\ $\eta_i = \sum_{j\neq i}\lambda_j$). Similarly, define the cone
\begin{equation*}
\widetilde{\Gamma}'_{k}=\{\lambda=(\lambda_{1}, \dots, \lambda_{n}) : \eta\in\widetilde{\Gamma}_{k}\}.
\end{equation*}
In the following, we collect some basic properties of $S_{k}(\eta)$ for $\lambda \in \widetilde{\Gamma}'_{k}$.

\begin{lemma}\label{prop-4}
Suppose that $\lambda=(\lambda_{1}, \dots, \lambda_{n})\in\widetilde{\Gamma}'_{k}$, then $S_{k}^{\frac{1}{k}}(\eta)$ is concave for $\eta=(\eta_1, \dots, \eta_n)$. Thus, for any $(\xi_{1}, \dots, \xi_{n})$ we have
\begin{equation*}
\sum\limits_{i, j=1}^{n}\frac{\partial^{2}[\frac{S_{k}(\eta)}{S_{l}(\eta)}]}{\partial\lambda_{i}\partial\lambda_{j}}\xi_{i}\xi_{j}
\leq(1-\frac{1}{k-l})\frac{\Big[\sum_{i}\frac{\partial[\frac{S_{k}(\eta)}{S_{l}(\eta)}]}{\partial\lambda_{i}}\xi_{i}\Big]^{2}}{\frac{S_{k}(\eta)}{S_{l}(\eta)}}.
\end{equation*}
\end{lemma}
\begin{proof}
A direct computation shows that
\begin{equation*}
\sum\limits_{i, j=1}^{n}\frac{\partial^{2}[\frac{S_{k}(\eta)}{S_{l}(\eta)}]}{\partial\lambda_{i}\partial\lambda_{j}}\xi_{i}\xi_{j}
=\sum\limits_{i, j, p, q=1}^{n}\frac{\partial^{2}[\frac{S_{k}(\eta)}{S_{l}(\eta)}]}{\partial\eta_{p}\partial\eta_{q}}\frac{\partial\eta_{p}}{\partial\lambda_{i}}\frac{\partial\eta_{q}}{\partial\lambda_{j}}\xi_{i}\xi_{j}.
\end{equation*}
By choosing new variables $\xi'=(\xi_{1}', \dots, \xi_{n}')$ given by
\begin{equation*}
\xi_{p}'=\sum_{i}\frac{\partial\eta_{p}}{\partial\lambda_{i}}\xi_{i},
\end{equation*}
and using Lemma~\ref{prop-3} (3), we complete the proof.
\end{proof}

\begin{lemma}\label{prop-5}
Suppose that $\lambda=(\lambda_{1}, \dots, \lambda_{n})\in\widetilde{\Gamma}_{k}$, if $1\leq l<k\leq n$, then
\begin{equation}\label{prop-5-ieq1}
\frac{S_{k-1}(\lambda)S_{l}(\lambda)}{C_{n+1}^{k-1}C_{n+1}^{l}}\geq\frac{S_{k}(\lambda)S_{l-1}(\lambda)}{C_{n+1}^{k}C_{n+1}^{l-1}}.
\end{equation}
If $0\leq l<k$ and $0\leq q<p$, then
\begin{equation}\label{prop-5-ieq2}
\Big(\frac{S_{k}(\lambda)}{S_{l}(\lambda)}\Big)^{\frac{1}{k-l}}\leq\Big(\frac{S_{p}(\lambda)}{S_{q}(\lambda)}\Big)^{\frac{1}{p-q}},
\end{equation}
for $k\geq p$ and $l\geq q$.
\end{lemma}
\begin{proof}
The proof can be found in \cite{RW25-2}.
\end{proof}

\begin{lemma}\label{prop-6}
Suppose that $\lambda=(\lambda_{1}, \dots, \lambda_{n})\in\widetilde{\Gamma}'_{k}$, and there is an ordering that $\lambda_{1}\geq\lambda_{2}\geq\dots\geq\lambda_{n}$, then we have\\
(1)~$\eta_{1}\leq\eta_{2}\leq\dots\leq\eta_{n}$, and $\eta_{n-k+2}>0$ for $2\leq k\leq n$;\\
(2)~$S_{k-1}(\eta|n-k+1)\geq c(n,k)S_{k-1}(\eta)$, $1\leq k<n$;\\
(3)~$\frac{\partial\big[\frac{S_{k}(\eta)}{S_{l}(\eta)}\big]}{\partial\eta_{1}}\geq\frac{\partial\big[\frac{S_{k}(\eta)}{S_{l}(\eta)}\big]}{\partial\eta_{2}}
\geq\dots\geq\frac{\partial\big[\frac{S_{k}(\eta)}{S_{l}(\eta)}\big]}{\partial\eta_{n}}$, $0\leq l<k\leq n$;\\
(4)~$\frac{\partial\big[\frac{S_{k}(\eta)}{S_{l}(\eta)}\big]}{\partial\lambda_{1}}\leq\frac{\partial\big[\frac{S_{k}(\eta)}{S_{l}(\eta)}\big]}{\partial\lambda_{2}}
\leq\dots\leq\frac{\partial\big[\frac{S_{k}(\eta)}{S_{l}(\eta)}\big]}{\partial\lambda_{n}}$, $0\leq l<k\leq n$;\\
(5)~$\forall~1\leq i\leq n$, $\frac{\partial\big[\frac{S_{k}(\eta)}{S_{l}(\eta)}\big]}{\partial\lambda_{i}}\geq 
c(n, k, l)\frac{1}{\tau}\sum\limits_{j=1}^{n}\frac{\partial\big[\frac{S_{k}(\eta)}{S_{l}(\eta)}\big]}{\partial\lambda_{j}}$, $0\leq l<k<n$;\\
(6)~$\sum\limits_{i=1}^{n}\frac{\partial\big[\frac{S_{k}(\eta)}{S_{l}(\eta)}\big]}{\partial\lambda_{i}}=(n\tau-1)\sum\limits_{i=1}^{n}\frac{\partial\big[\frac{S_{k}(\eta)}{S_{l}(\eta)}\big]}{\partial\eta_{i}}
\geq c(n, k, l)\tau f^{1-\frac{1}{k-l}}$ for equation \eqref{Hessian quo-2}, $0\leq l<k\leq n$.
\end{lemma}
\begin{proof}
Both (1) and (2) follow directly from Lemma \ref{prop-3}. From \eqref{prop-7-ieq2} and the concavity of $[\frac{S_{k}(\eta)}{S_{l}(\eta)}]^{1/k-l}$ in $\widetilde{\Gamma}_{k}$, assertion (3) follows directly. 

By the definition of $\eta$, a straightforward calculation gives
\begin{equation*}
\frac{\partial\big[\frac{S_{k}(\eta)}{S_{l}(\eta)}\big]}{\partial\lambda_{i}}
=\frac{\partial\big[\frac{S_{k}(\eta)}{S_{l}(\eta)}\big]}{\partial\eta_{p}}\frac{\partial\eta_{p}}{\partial\lambda_{i}}
=\tau\sum\limits_{j=1}^{n}\frac{\partial\big[\frac{S_{k}(\eta)}{S_{l}(\eta)}\big]}{\partial\eta_{j}}-\frac{\partial\big[\frac{S_{k}(\eta)}{S_{l}(\eta)}\big]}{\partial\eta_{i}},
\end{equation*}
combining (3) and $\tau\geq1$ yields (4). 

Next, we prove (5) similarly to the proof of Lemma~11 in \cite{RW25-2}. A direct computation yields
\begin{equation}\label{prop-6 ineq-1}
\begin{aligned}
\sum\limits_{j=1}^{n}\frac{\partial\big[\frac{S_{k}(\eta)}{S_{l}(\eta)}\big]}{\partial\lambda_{j}}=&~(n\tau-1)\sum\limits_{j=1}^{n}\frac{\partial\big[\frac{S_{k}(\eta)}{S_{l}(\eta)}\big]}{\partial\eta_{j}}\\
=&~(n\tau-1)\sum\limits_{j=1}^{n}\frac{S_{k-1}(\eta|j)S_{l}(\eta)-S_{k}(\eta)S_{l-1}(\eta|j)}{S_{l}^{2}(\eta)}\\
=&~(n\tau-1)\frac{1}{S_{l}^{2}(\eta)}\Big[(n-k+2)S_{k-1}(\eta)S_{l}(\eta)-\sigma_{k-1}(\eta)S_{l}(\eta)\\&-(n-l+1)S_{k}(\eta)S_{l-1}(\eta)-\alpha\sigma_{l-2}(\eta)S_{k}(\eta)\Big]\\
\leq&~(n\tau-1)(n-k+2)\frac{S_{k-1}(\eta)}{S_{l}(\eta)}.
\end{aligned}
\end{equation}
By (4), it suffices to prove the case $i=1$, the other cases then follow directly. From Lemma \ref{N-M}, we have
\begin{equation*}
\begin{aligned}
\frac{\partial\big[\frac{S_{k}(\eta)}{S_{l}(\eta)}\big]}{\partial\lambda_{1}}
=&~\tau\sum\limits_{j=1}^{n}\frac{S_{k-1}(\eta|j)S_{l}(\eta)-S_{k}(\eta)S_{l-1}(\eta|j)}{S_{l}^{2}(\eta)}-\frac{S_{k-1}(\eta|1)S_{l}(\eta)-S_{k}(\eta)S_{l-1}(\eta|1)}{S_{l}^{2}(\eta)}\\
=&~\tau\sum\limits_{j=1}^{n}\frac{S_{k-1}(\eta|j)S_{l}(\eta|j)-S_{k}(\eta|j)S_{l-1}(\eta|j)}{S_{l}^{2}(\eta)}-\frac{S_{k-1}(\eta|1)S_{l}(\eta|1)-S_{k}(\eta|1)S_{l-1}(\eta|1)}{S_{l}^{2}(\eta)}\\
\geq&~\frac{(n+1)(k-l)}{k(n-l+1)}\bigg[\tau\sum\limits_{j=2}^{n}\frac{S_{k-1}(\eta|j)S_{l}(\eta|j)}{S_{l}^{2}(\eta)}+(\tau-1)\frac{S_{k-1}(\eta|1)S_{l}(\eta|1)}{S_{l}^{2}(\eta)}\bigg]\\
\geq&~\frac{(n+1)(k-l)}{k(n-l+1)}\frac{S_{k-1}(\eta|2)S_{l}(\eta|2)}{S_{l}^{2}(\eta)}.
\end{aligned}
\end{equation*}
If $0\leq l<k< n$, we know $n-l \geq n-k+1 \geq 2$. Combining this with (2) and (3) gives 
\begin{equation}\label{prop-6 ineq-2}
\begin{aligned}
S_{k-1}(\eta|2)\geq&~S_{k-1}(\eta|n-k+1) \geq c(n,k)S_{k-1}(\eta)\\
S_{l}(\eta|2)&\geq S_{l}(\eta|n-l)\geq c(n,l)S_{l}(\eta)
\end{aligned}
\end{equation}
Together with \eqref{prop-6 ineq-1} and \eqref{prop-6 ineq-2}, this yields
\begin{equation*}
\begin{aligned}
\frac{\partial\big[\frac{S_{k}(\eta)}{S_{l}(\eta)}\big]}{\partial\lambda_{1}}
\geq c(n, k, l)\frac{S_{k-1}(\eta)}{S_{l}(\eta)}\geq c(n, k, l)\frac{1}{\tau}\sum\limits_{j=1}^{n}\frac{\partial\big[\frac{S_{k}(\eta)}{S_{l}(\eta)}\big]}{\partial\lambda_{j}}.
\end{aligned}
\end{equation*}
This completes the proof of (5).

Finally, we prove (6). By direct calculation, we obtain
\begin{equation*}
\begin{aligned}
\sum\limits_{j=1}^{n}\frac{\partial\big[\frac{S_{k}(\eta)}{S_{l}(\eta)}\big]}{\partial\eta_{j}}
=&~\frac{1}{S_{l}^{2}(\eta)}\big[(n-k+2)S_{k-1}(\eta)S_{l}(\eta)+\alpha\sigma_{k-2}(\eta)S_{l}(\eta)\\&-(n-l+1)S_{k}(\eta)S_{l-1}(\eta)-\alpha\sigma_{l-2}(\eta)S_{k}(\eta)\big]\\
=&~\frac{1}{S_{l}^{2}(\eta)}\big[(n-k+1)\sigma_{k-1}\sigma_{l}-(n-l+1)\sigma_{k}\sigma_{l-1}+\alpha(n-k+2)\sigma_{k-2}\sigma_{l}\\
&-\alpha(n-l+1)\sigma_{k-1}\sigma_{l-1}+\alpha(n-k+1)\sigma_{k-1}\sigma_{l-1}-\alpha(n-l+2)\sigma_{k}\sigma_{l-2}\\
&+\alpha^{2}(n-k+2)\sigma_{k-2}\sigma_{l-1}-\alpha^{2}(n-l+2)\sigma_{k-1}\sigma_{l-2}\big]\\ 
\geq&~\frac{1}{S_{l}^{2}(\eta)}\big[\frac{k-l}{k}(n-k+1)\sigma_{k-1}\sigma_{l}+\alpha\frac{k-l-1}{k-1}(n-k+2)\sigma_{k-2}\sigma_{l}\\
&+\alpha\frac{k-l+1}{k}(n-k+1)\sigma_{k-1}\sigma_{l-1}+\alpha^{2}\frac{k-l}{k-1}(n-k+2)\sigma_{k-2}\sigma_{l-1}\big]\\
\geq&~c(n,k,l)\frac{S_{k-1}(\eta)}{S_{l}(\eta)}.
\end{aligned}
\end{equation*}
Combing with Lemma \ref{prop-5} and (5), we obtain
\begin{equation*}
\begin{aligned}
\sum\limits_{i=1}^{n}\frac{\partial\big[\frac{S_{k}(\eta)}{S_{l}(\eta)}\big]}{\partial\lambda_{i}}
=&~(n\tau-1)\sum\limits_{i=1}^{n}\frac{\partial\big[\frac{S_{k}(\eta)}{S_{l}(\eta)}\big]}{\partial\eta_{i}}\geq c(n,k,l)\tau\frac{S_{k-1}(\eta)}{S_{l}(\eta)}\\
&\geq c(n,k,l)\tau\big(\frac{S_{k}(\eta)}{S_{l}(\eta)}\big)^{1-\frac{1}{k-l}}\geq c(n,k,l)\tau f^{1-\frac{1}{k-l}}.
\end{aligned}
\end{equation*}
This finishes the proof.
\end{proof}

We also need the following well-known result (see, e.g., \cite{Ball84}).

\begin{lemma}\label{prop-7}
If $W=(w_{ij})$ is a real symmetric matrix, $\lambda_{i}=\lambda_{i}(W)$ is one of the eigenvalues $(i=1, \dots, n)$ and $F=F(W)=f(\lambda(W))$ is a symmetric function of $\lambda_{1}, \dots, \lambda_{n}$, then for any real symmetric matrix $A=(a_{ij})$, we have
\begin{equation}\label{prop-7-ieq1}
\frac{\partial^{2}F}{\partial w_{ij}\partial w_{st}}a_{ij}a_{st}=\frac{\partial^{2}f}{\partial \lambda_{p}\partial \lambda_{q}}a_{pp}a_{qq}
+2\sum\limits_{p<q}\frac{\frac{\partial f}{\partial \lambda_{p}}-\frac{\partial f}{\partial \lambda_{q}}}{\lambda_{p}-\lambda_{q}}a_{pq}^{2}.
\end{equation}
Moreover, if $f$ is concave and $\lambda_{1}\geq\lambda_{2}\geq\dots\geq\lambda_{n}$, we have
\begin{equation}\label{prop-7-ieq2}
\frac{\partial f}{\partial \lambda_{1}}(\lambda)\leq\frac{\partial f}{\partial \lambda_{2}}(\lambda)\leq\dots\leq\frac{\partial f}{\partial \lambda_{n}}(\lambda).
\end{equation}
\end{lemma}
\begin{proof}
Lemma 3.2 of \cite{GLM18} (see also \cite{BAn94}) gives the proof of \eqref{prop-7-ieq1}, while Lemma 2.2 of \cite{BAn94} proves \eqref{prop-7-ieq2}.
\end{proof}

\section{Proof of Theorem \ref{maintheorem1}}  \label{S3}

In this section, we establish the interior a priori Hessian estimate for \eqref{Hessian quo-2}, following the method similar to \cite{CDH23, RW25-2}.

Assume that $u\in C^{4}(B_{r}(o))$ is a solution of the equation \eqref{Hessian quo-2} with $\lambda(\nabla^{2}u)\in\Gamma_{k-1}'$. For any complete Riemannian manifold $M^n$, one can choose a suitable point $o$ as the origin. Subsequently, we apply the maximum principle on geodesic ball $B_{r}(o)$. Since $B_{r}(o)$ is a compact domain, there exists a positive constant $K$ depending only on the Riemann curvature tensor such that $\underline{K}_{B_{r}(o)}\geq-K$.

For any point $x$ on $M^{n}$, we choose local coordinates $(\xi^{1},\dots,\xi^{n})$. For convenience, let $\partial/\partial \xi^{i}$ be the corresponding coordinate vector fields on $M^{n}$. Denote by $\zeta(x)=\zeta(\max\lambda(\nabla^2u))=(\zeta_{1}(x),\cdots,\zeta_{n}(x))$ the continuous eigenvector field of $\nabla^2u(x)$ corresponding to the maximum eigenvalue. 

Consider the auxiliary function
\begin{equation*}
\phi(x)=s(x)h(\frac{|\nabla u(x)|^2}{2})u_{\zeta\zeta}
\end{equation*}
on $B_{r}(o)$, where $s(x)=r^{2}-\rho^{2}(x)$ (i.e. $\rho(x)=\rho(x,o)$) and $h(t)=(1-\frac{t}{A})^{-\frac{1}{3}}$ with $A=\sup|\nabla u|^{2}$. Then we have
\begin{equation*}
h'(t)=\frac{1}{3A}(1-\frac{t}{A})^{-\frac{4}{3}},\quad\quad h''(t)=\frac{4}{9A^{2}}(1-\frac{t}{A})^{-\frac{7}{3}}.
\end{equation*}

Suppose that $\phi(x)$ attains its maximum value in $B_{r}(o)$ at $x_{0}$.
If $x_{0}\in\partial B_{r}(o)$, then $\rho(x_{0})=r$ and hence $\phi(x_{0})=0$, which contradicts that $\phi$ attains its maximum at $x_{0}$. Thus $x_0$ must lie in the interior of $B_{r}(o)$.
In the following we distinguish two cases according to the position of $x_{0}$. 

We first consider the case $x_{0} \notin \operatorname{Cut}(o)$
(Note that when the sectional curvature of $M^{n}$ is nonpositive, the cut locus $\operatorname{Cut}(o)=\emptyset$; hence the case $x_{0}\notin\operatorname{Cut}(o)$ is automatically satisfied for all $x_{0}$. 
For the uniformity of the proof, we keep the same case distinction as in the general setting).

Choosing a local orthonormal frame field at $x_{0}$ such that $g_{ij}(x_{0})=\delta_{ij}$, and then rotating the coordinate axes, we can diagonalize the Hessian matrix $\nabla^{2}u=(u_{ij})$,
\begin{equation*}
u_{ij}(x_{0})=u_{ii}(x_{0})\delta_{ij},\quad\quad u_{11}(x_{0})\geq u_{22}(x_{0})\geq\cdots\geq u_{nn}(x_{0}).
\end{equation*}
From the above, we have $\zeta(x_0) = (1,0,\dots,0)$. Denote $\lambda_{i} = u_{ii}(x_0)$ and $\lambda = (\lambda_1,\cdots,\lambda_n)$, and then $\lambda_1 \geq \lambda_2 \geq \cdots \geq \lambda_n$. Moreover, the test function
\begin{equation}\label{3-eq1}
\varphi=\log s+\log h(\frac{|\nabla u(x)|^2}{2})+\log u_{11}
\end{equation}
also attains local maximum at $x_{0}$. 

Differentiating $\varphi$ at $x_{0}$ once yields
\begin{equation}\label{3-eq2}
0=\varphi_{i}=\frac{s_{i}}{s}+\frac{h'}{h}\sum_{k=1}^{n}u_{k}u_{ki}+\frac{u_{11i}}{u_{11}},
\end{equation}
therefore,
\begin{equation}\label{3-ieq3}
\frac{u_{11i}^{2}}{u_{11}^{2}}\leq2\frac{s_{i}^{2}}{s^{2}}+2\frac{(h')^{2}}{h^{2}}u_{i}^{2}u_{ii}^{2},
\end{equation}
Differentiating $\varphi$ at $x_{0}$ twice results in
\begin{equation}\label{3-ieq4}
\begin{aligned}
0\geq\varphi_{ii}=&~\frac{s_{ii}}{s}-\frac{s_{i}^{2}}{s^{2}}+\frac{h''h-(h')^{2}}{h^{2}}\sum_{k=1}^{n}u_{k}u_{ki}\sum_{l=1}^{n}u_{l}u_{li}\\
&+\frac{h'}{h}\sum_{k=1}^{n}(u_{ki}u_{ki}+u_{k}u_{kii})+\frac{u_{11ii}}{u_{11}}-\frac{u_{11i}^{2}}{u_{11}^{2}}\\
\geq&~\frac{s_{ii}}{s}-3\frac{s_{i}^{2}}{s^{2}}+\frac{h''h-3(h')^{2}}{h^{2}}u_{i}^{2}u_{ii}^{2}+\frac{h'}{h}\bigg(u_{ii}^{2}+\sum_{k=1}^{n}u_{k}u_{kii}\bigg)+\frac{u_{11ii}}{u_{11}}\\
\geq&~\frac{s_{ii}}{s}-3\frac{s_{i}^{2}}{s^{2}}+\frac{h'}{h}\bigg(u_{ii}^{2}+\sum_{k=1}^{n}u_{k}u_{kii}\bigg)+\frac{u_{11ii}}{u_{11}},
\end{aligned}
\end{equation}
where the last inequality follows from the identity $h''h-3(h')^{2}=\frac{1}{9A^{2}}(1-\frac{t}{A})^{-\frac{8}{3}}>0$.

Denote 
\begin{equation*}
F^{ij} = \frac{\partial \frac{S_{k}(\eta)}{S_{l}(\eta)}}{\partial u_{ij}}=
\begin{cases}
\frac{\partial \frac{S_{k}(\eta)}{S_{l}(\eta)}}{\partial \lambda_{i}},\quad i=j,\\
0,\quad\quad\quad i\neq j,
\end{cases}
\end{equation*}
and
\begin{equation*}
F^{ij,rs} = \frac{\partial^{2} \frac{S_{k}(\eta)}{S_{l}(\eta)}}{\partial u_{ij}\partial u_{rs}}.
\end{equation*}
Thus, we obtain
\begin{equation}\label{3-eq5}
\sum_{i=1}^{n}F^{ii}u_{ii1} = \sum_{i,j=1}^{n}F^{ij}u_{ij1} = \partial_{1}f+\partial_{u}fu_{1}+\sum_{i=1}^{n}\partial_{u_{i}}fu_{i1}.
\end{equation}
where $\partial_{i}f$ denotes the partial derivative of $f$ with respect to the $i$-th component $\xi_{i}$ of its first variable $x$, $\partial_{u}f$ denotes the partial derivative of $f$ with respect to $u$, and $\partial_{u_{i}}f$ denotes the partial derivative of $f$ with respect to $\nabla u$. Differentiating equation \eqref{Hessian quo-2} twice gives
\begin{equation}\label{3-eq6}
F^{ij,rs}u_{ijp}u_{rsp}+F^{ii}u_{iipp}=f_{pp}.
\end{equation}
From the definition of $f$, \eqref{3-eq7-1} and \eqref{3-eq2}, we have the following estimate
\begin{equation}\label{3-ieq8}
\begin{aligned}
|f_{11}| =&~|\partial^{2}_{11}f+\partial^{2}_{\xi_{1}u}f\frac{\partial u}{\partial \xi_{1}}+\sum_{i=1}^{n}\partial^{2}_{\xi_{1}u_{i}}f\frac{\partial u_{i}}{\partial \xi_{1}}+\partial^{2}_{u\xi_{1}}f\frac{\partial u}{\partial \xi_{1}}+\partial^{2}_{uu}f(\frac{\partial u}{\partial \xi_{1}})^{2}\\
&+\sum_{i=1}^{n}\partial^{2}_{uu_{i}}f\frac{\partial u_{i}}{\partial \xi_{1}}\frac{\partial u_{i}}{\partial \xi_{1}}+\partial_{u}f\frac{\partial^{2} u}{\partial \xi_{1}^{2}}+\sum_{i=1}^{n}\partial^{2}_{u_{i}\xi_{1}}f\frac{\partial u_{i}}{\partial \xi_{1}}+\sum_{i=1}^{n}\partial^{2}_{u_{i}u}f\frac{\partial u_{i}}{\partial \xi_{1}}\frac{\partial u}{\partial \xi_{1}}\\
&+\sum_{i=1}^{n}\sum_{j=1}^{n}\partial^{2}_{u_{i}u_{j}}f\frac{\partial u_{i}}{\partial \xi_{1}}\frac{\partial u_{j}}{\partial \xi_{1}}+\sum_{i=1}^{n}\partial_{u_{i}}f\frac{\partial^{2} u_{i}}{\partial \xi_{1}^{2}}|\\
\leq&~C(1+u_{11}+u_{11}^{2})+|\sum_{i=1}^{n}\partial_{u_{i}}fR^{m}_{1i1}u_{m}|+|\sum_{i=1}^{n}\partial_{u_{i}}fu_{11i}|\\
\leq&~C(1+u_{11}+u_{11}^{2}),
\end{aligned}
\end{equation} 
where we used the fact that the components of the Riemannian curvature tensor $R$ are bounded on $B_{r}(o)$, $f_{11}$ denotes the second covariant derivatives of $f$, and $C$ is a positive constant depending only on $n$, $r$, $||u||_{C^{1}}$, $||R_{ijk}^{l}||_{C^{0}}$, $||f||_{C^{2}(B_{r}(o)\times\mathbb{R}\times\mathbb{R}^{n})}$. Combining Lemma \ref{prop-4} with \eqref{3-eq2}, \eqref{3-eq5}, \eqref{3-eq6} and \eqref{3-ieq8}, assuming $u_{11}(x_{0})\geq1+\sup|\nabla u|$, we obtain
\begin{equation}\label{3-ieq9}
\begin{aligned}
\sum_{i=1}^{n}F^{ii}u_{ii11}&=\sum_{i,j=1}^{n}F^{ij}u_{ij11}\geq-\sum_{i,j,r,s=1}^{n}F^{ij,rs}u_{ij1}u_{rs1}-Cu_{11}^{2}-C\\
&\geq-\left(1-\frac{1}{k-l}\right)\frac{(\sum_{i=1}F^{ii}u_{ii1})^{2}}{f}-Cu_{11}^{2}-C\geq-Cu_{11}^{2}-C.
\end{aligned}
\end{equation}
From \eqref{3-eq7-1}, \eqref{3-eq7-2}, \eqref{3-ieq3}, \eqref{3-eq5} and \eqref{3-ieq9}, it follows that
\begin{equation}\label{3-ieq11}
\begin{aligned}
0\geq&\sum_{i=1}^{n}F^{ii}\varphi_{ii}\\
\geq&\sum_{i=1}^{n}F^{ii}\left(\frac{s_{ii}}{s}-3\frac{s_{i}^{2}}{s^{2}}\right)+\frac{h'}{h}\sum_{i=1}^{n}F^{ii}u_{ii}^{2}+\frac{h'}{h}\sum_{i=1}^{n}F^{ii}\sum_{k=1}^{n}\bigg(u_{k}u_{iik}+u_{k}R_{iki}^{m}u_{m}\bigg)\\
&+\sum_{i=1}^{n}F^{ii}\left(\frac{u_{ii11}}{u_{11}}-2R_{ii1}^{1}-2R_{1i1}^{i}\frac{u_{ii}}{u_{11}}-\frac{u_{j}}{u_{11}}(R_{ii1}^{j})_{,\xi_{1}}-\frac{u_{j}}{u_{11}}(R_{1i1}^{j})_{,\xi_{i}}\right)\\
\geq&\sum_{i=1}^{n}F^{ii}\left(\frac{s_{ii}}{s}-3\frac{s_{i}^{2}}{s^{2}}\right)+\frac{h'}{h}\sum_{i=1}^{n}F^{ii}u_{ii}^{2}-\sum_{i=1}^{n}F^{ii}(C\frac{h'}{h}+C)\\
&+\frac{h'}{h}\sum_{k=1}^{n}u_{k}\big(\partial_{k}f+\partial_{u}fu_{k}+\sum_{i=1}^{n}\partial_{u_{i}}fu_{ik}\big)-Cu_{11}-C\\
\geq&\sum_{i=1}^{n}F^{ii}\left(\frac{s_{ii}}{s}-3\frac{s_{i}^{2}}{s^{2}}\right)+\frac{1}{3A}F^{11}u_{11}^{2}-C\sum_{i=1}^{n}F^{ii}-Cu_{11}-C,
\end{aligned}
\end{equation}
where $C$ is a positive constant depending on $n$, $k$, $l$, $m_{1}$, $||R_{ijk}^{l}||_{C^{1}}$ and $||f||_{C^{2}}$.

If $x_{0}=o$, $\rho(x)$ is not differential at $o$ at this time; however, a standard fact in Riemannian geometry gives $\nabla(\rho^{2}) = 0$ and $\nabla^{2}(\rho^{2}) = 2g$ at $o$. Then we have
\begin{equation*}
s_{i}=-\nabla_{\frac{\partial}{\partial\xi^{i}}} (\rho^{2})=0, \quad s_{ii}=\nabla^{2}s(\frac{\partial}{\partial\xi^{i}},\frac{\partial}{\partial\xi^{i}})= -\nabla^{2}(\rho^{2})(\frac{\partial}{\partial\xi^{i}},\frac{\partial}{\partial\xi^{i}})= -2g(\frac{\partial}{\partial\xi^{i}},\frac{\partial}{\partial\xi^{i}})=-2.
\end{equation*}
If $x_{0}\neq o$, $\rho(x)$ is differentiable at $x_{0}$, $B_{r}(o)$ is a geodesic ball of radius $r$ with $\underline{K}_{B_{r}(o)}\geq-K$. on it. Then by Lemma \ref{Hess-compar}, we have on $B_{r}(o)\backslash \operatorname{Cut}(o)\cup\{o\}$
\begin{equation*}
s_{i}=-2\rho\rho_{i}\geq-2r\sup|\nabla\rho|,
\end{equation*}
\begin{equation*}
s_{ii}=-2\rho\rho_{ii}-2\rho_{i}^{2} \geq -2\sqrt{K}\rho \coth(\sqrt{K} \rho)(1-\rho_{i}^{2})-2\rho_{i}^{2}\geq -2\sqrt{K}r \coth(\sqrt{K} r).
\end{equation*}
Therefore, combining the above two cases, we obtain
\begin{equation}\label{3-ieq11-2}
\frac{s_{ii}}{s}-3\frac{s_{i}^{2}}{s^{2}}\geq-\frac{C}{s}-\frac{C}{s^{2}},
\end{equation}
where $C$ is a positive constant depending on $r$, $||R_{ijk}^{l}||_{C^{0}}$, $\sup|\nabla\rho|$.

Lemma \ref{prop-6} (5) and (6) implies that
\begin{equation}\label{3-ieq12}
F^{11}\geq c(n, k, l)\frac{1}{\tau}\sum\limits_{i=1}^{n}F^{ii}\geq C,
\end{equation}
where $C$ is a positive constant depending on $n$, $k$, $l$, $m_{1}$. Plugging \eqref{3-ieq11-2} and \eqref{3-ieq12} into \eqref{3-ieq11}, we have at $x_{0}$,
\begin{equation}\label{3-ieq13}
\begin{aligned}
0\geq&\sum_{i=1}^{n}F^{ii}\bigg(-\frac{C}{s}-\frac{C}{s^{2}}\bigg)+\frac{C}{A}u_{11}^{2}\sum\limits_{i=1}^{n}F^{ii}-C\sum_{i=1}^{n}F^{ii}-Cu_{11}-C.
\end{aligned}
\end{equation}
If $s^{2}u_{11}^{2}(x_{0})$ is sufficiently large, combined with the bound $\sum_{i}F^{ii}\geq C$, the right-hand side of the inequality becomes positive, yielding a contradiction. Hence 
\begin{equation}\label{3-ieq14}
s(x_{0})u_{11}(x_{0}) \leq C(1+\sup|\nabla u|),
\end{equation}
where $C$ is a positive constant depending on $n$, $k$, $l$, $m_{1}$, $r$, $\sup|\nabla\rho|$, $||R_{ijk}^{l}||_{C^{1}}$ and $||f||_{C^{2}}$. 

Next, we consider the case $x_{0} \in \operatorname{Cut}(o)$. Let $\gamma(\theta)$ be a least distance geodesic with speed $1$, and $\gamma(0)=o$, $\gamma(L)=x_{0}$. Then for any small $\varepsilon>0$, let $x_{1}=\gamma(\varepsilon)$, we have $\operatorname{dist}_{M^{n}}(x_{1}, o)=\varepsilon$, and $x_{0}$ is not a cut point of $x_{1}$ with $\rho_{1}(x)=\operatorname{dist}_{M^{n}}(x, x_{1})$. Define a new distance function
\begin{equation*}
\bar{\rho}(x_{0})=\rho_{1}(x_{0})+\varepsilon,
\end{equation*}
and by triangle inequality we obtain
\begin{equation*}
\rho_{1}(x)+\varepsilon\geq\operatorname{dist}_{M^{n}}(x, o)=\rho(x).
\end{equation*}
Thus, consider the test function
\begin{equation*}
\bar{\phi}(x)=\big(r^{2}-(\rho_{1}(x)+\varepsilon)^{2}\big)h(\frac{|\nabla u(x)|^2}{2})u_{\zeta\zeta}
\end{equation*}
at a neighbourhood of $x_{0}$. Then we obtain
\begin{equation*}
\bar{\phi}(x)\leq\phi(x)\leq\phi(x_{0})\leq\bar{\phi}(x_{0}),
\end{equation*}
which implies that $x_{0}$ is still a maximum point of $\bar{\phi}$. Then, by the same calculations and application of the Hessian comparison theorem as in the case $x_0 \notin \operatorname{Cut}(o)$, we can similarly obtain \eqref{3-ieq14}.

In summary, we obtain
\begin{equation*}
u_{\zeta\zeta}(o)\leq\frac{h(\frac{1}{2}|\nabla u(x_{0})|^{2})}{h(\frac{1}{2}|\nabla u(0)|^{2})}\frac{\rho(x_{0})^{2}}{r^{2}}u_{\zeta\zeta}(x_{0})\leq C(1+\sup|\nabla u|).
\end{equation*}
The proof is complete.

\section{Proof of Theorem \ref{maintheorem2}}  \label{S4}

In this section we prove Theorem \ref{maintheorem2}, continuing the computations from Section \ref{S3}. 

Suppose $u\in C^{4}(\Omega\cap C^{2}(\overline{\Omega}))$ is a solution of equation \eqref{Hessian quo-3} with $\eta\in\widetilde{\Gamma}_{k}$. Without loss of generality, we assume $u<0$ in $\Omega$. 
Consider the following test function
\begin{equation*}
\tilde{P}(x)=\log(-u)+\log \lambda_{max}(x)+\frac{a}{2}|\nabla u|^{2},
\end{equation*}
where $\lambda_{max}(x)$ is the biggest eigenvalue of the Hessian matrix $u_{ij}$. Suppose that $\tilde{P}(x)$ attains its maximum value in $\Omega$ at $x_{0}$. Choosing a local orthonormal frame field at $x_{0}$ such that $g_{ij}(x_{0})=\delta_{ij}$, and then rotating the coordinate axes, we can diagonalize the Hessian matrix $\nabla^{2}u=(u_{ij})$,
\begin{equation*}
u_{ij}(x_{0})=u_{ii}(x_{0})\delta_{ij},\quad\quad u_{11}(x_{0})\geq u_{22}(x_{0})\geq\cdots\geq u_{nn}(x_{0}).
\end{equation*}

On $\Omega$, define a new function
\begin{equation*}
P(x)=\log(-u)+\log u_{11}(x)+\frac{a}{2}|\nabla u|^{2},
\end{equation*}
which also attains a local maximum at $x_{0}$. Differentiating $P$ at $x_{0}$ once yields
\begin{equation}\label{4-eq1}
\frac{u_{i}}{u}+\frac{u_{11i}}{u_{11}}+au_{ii}u_{i}=0,
\end{equation}
and differentiating $P$ at $x_{0}$ twice,
\begin{equation}\label{4-ieq2}
\frac{u_{ii}}{u}-\frac{u_{i}^{2}}{u^{2}}+\frac{u_{11ii}}{u_{11}}-\frac{u_{11i}^{2}}{u_{11}^{2}}+a\sum_{p=1}^{n}u_{pii}u_{i}+au_{ii}^{2}\leq0.
\end{equation}
Thus, at $x_{0}$
\begin{equation}\label{4-ieq3}
\begin{aligned}
0\geq&~F^{ii}P_{ii}\\
\geq&~\frac{F^{ii}u_{ii}}{u}-\frac{F^{ii}u_{i}^{2}}{u^{2}}+\frac{F^{ii}u_{11ii}}{u_{11}}-\frac{F^{ii}u_{11i}^{2}}{u_{11}^{2}}+a\sum_{p=1}^{n}F^{ii}u_{pii}u_{p}+aF^{ii}u_{ii}^{2}.
\end{aligned}
\end{equation}
From \eqref{3-eq7-1}, \eqref{3-ieq8} and \eqref{4-eq1}, it follows that
\begin{equation}\label{4-ieq3-2}
\begin{aligned}
|f_{11}|\leq&~C(1+u_{11}+u_{11}^{2})+|\sum_{i=1}^{n}\partial_{u_{i}}fR^{m}_{1i1}u_{m}|+|\sum_{i=1}^{n}\partial_{u_{i}}fu_{11i}|\\
\leq&~C(1+u_{11}+u_{11}^{2})+Cau_{11}^{2}-\frac{C}{u}u_{11}.
\end{aligned}
\end{equation}
By \eqref{3-eq7-1}, \eqref{3-eq7-2}, \eqref{3-ieq9} and \eqref{4-ieq3-2}, we obtain
\begin{equation}\label{4-ieq4}
\begin{aligned}
F^{ii}u_{11ii}\geq&~F^{ii}u_{ii11}-\sum_{i=1}^{n}F^{ii}(Cu_{11}+C)\\
\geq&-(Ca+C)u_{11}^{2}-Cu_{11}+\frac{C}{u}u_{11}-C-\sum_{i=1}^{n}F^{ii}(Cu_{11}+C),
\end{aligned}
\end{equation}
where $C$ is a positive constant depending on $n$, $k$, $l$, $||u||_{C^{1}}$, $||R_{ijk}^{l}||_{C^{0}}$ and $||f||_{C^{2}}$. Applying the Cauchy–Schwarz inequality to \eqref{4-eq1} yields
\begin{equation}\label{4-ieq5}
-\frac{F^{ii}u_{11i}^{2}}{u_{11}^{2}}\geq-\frac{C}{u^{2}}\sum_{i=1}^{n}F^{ii}-Ca^{2}F^{ii}u_{ii}^{2}
\end{equation}
where $C$ is a positive constant depending on $n$ and $||u||_{C^{1}}$. By \eqref{3-eq7-1} and \eqref{3-eq5}, one drives
\begin{equation}\label{4-ieq6}
a\sum_{p=1}^{n}F^{ii}u_{pii}u_{p}=a\sum_{p=1}^{n}F^{ii}u_{iip}u_{p}+a\sum_{p=1}^{n}F^{ii}u_{m}u_{p}R_{ipi}^{m}\geq-Ca\sum_{i=1}^{n}F^{ii}-Cau_{11}-Ca.
\end{equation}
Choose $a$ sufficiently small. Plugging \eqref{4-ieq4}, \eqref{4-ieq5} and \eqref{4-ieq6} into \eqref{4-ieq7}, assuming $u_{11}(x_{0})\geq1$, we obtain
\begin{equation}\label{4-ieq7}
\begin{aligned}
0\geq&~\frac{Cu_{11}}{u}\sum_{i=1}^{n}F^{ii}-\frac{C}{u^{2}}\sum_{i=1}^{n}F^{ii}+(a-Ca^{2})F^{ii}u_{ii}^{2}+\frac{C}{u}\\
&-(Ca+C)\sum_{i=1}^{n}F^{ii}-(Ca+C)u_{11}-Ca-C\\
\geq&~\frac{1}{C}F^{11}u_{11}^{2}+\bigg(\frac{Cu_{11}}{u}-\frac{C}{u^{2}}-C\bigg)\sum_{i=1}^{n}F^{ii}-Cu_{11}-C+\frac{C}{u}
\end{aligned}
\end{equation}
where $C$ is a positive constant depending on $n$, $k$, $l$, $||R_{ijk}^{l}||_{C^{1}}$, $||f||_{C^{2}}$ and $||u||_{C^{1}}$. 
Lemma \ref{prop-6} (5) (6) implies that  $F^{11}\geq C\sum_{i}F^{ii}\geq C$.
Then at $x_{0}$ we obtain
\begin{equation*}
0\geq \bigg(\frac{1}{C}u_{11}^{2}+\frac{Cu_{11}}{u}-\frac{C}{u^{2}}-C\bigg)\sum_{i=1}^{n}F^{ii}-Cu_{11}-C+\frac{C}{u}
\end{equation*}
If $(-u)^{2}u_{11}^{2}(x_{0})$ is sufficiently large, combined with the bound $\sum_{i}F^{ii}\geq C$, the right-hand side of the inequality becomes positive, yielding a contradiction. So we complete the proof.

\section{Proof of Theorem \ref{maintheorem3}}  \label{S5}

In this section, we will use the idea in \cite{CTX21, RW25-2} to give the proof of Theorem \ref{maintheorem3}. 
Assume that $u\in C^{4}(\Omega\cap C^{2}(\overline{\Omega}))$ is a solution of equation \eqref{Hessian quo-3} with $\eta\in\widetilde{\Gamma}_{k}$. Without loss of generality, we assume $u<0$ in $\Omega$.  

Since $\overline{\Omega}$ is a compact domain, there exists a constant $\overline{K}_{\Omega}$ depending only on the Riemann curvature tensor.
Choose a point $p_{0} \in \Omega$; if $\overline{K}_{\Omega}>0$, we further require that $\overline{\Omega}$ is contained in the geodesic ball $B_{r}(p_{0})$ with $r < \pi/2\sqrt{\overline{K}_{\Omega}}$. For convenience, we just write $K:=\overline{K}_{\Omega}$ in the calculations.
Define a function $\rho$ on $\Omega$ as follows
\begin{equation}\label{def-rho}
\rho(x) = \operatorname{dist}_{M^{n}}(p_{0}, x),\quad \forall x\in\Omega,
\end{equation}
where $\operatorname{dist}_{M^{n}}(p_{0}, x)$ measures the geodesic distance between $p_{0}$ and $x$.

For convenience, we set $U[u]=(\tau\Delta u)I - \nabla^{2}u$ with $\tau\geq 1$, and introduce the following notations
\begin{equation*}
F(U)=\left[\frac{S_{n}(U)}{S_{l}(U)}\right]^{\frac{1}{n-l}},\quad T(\nabla^{2}u)=F(U),
\end{equation*}
\begin{equation*}
F^{ij}=\frac{\partial F}{\partial U_{ij}},\quad F^{ij,rs}=\frac{\partial^{2} F}{\partial U_{ij}\partial U_{rs}},\quad T^{ii}=\tau\sum_{i=1}^{n}F^{jj}-F^{ii}.
\end{equation*}
Consider the following test function
\begin{equation*}
\tilde{P}(x)=\beta\log(-u)+\log \lambda_{max}(x)+\frac{a}{2}|\nabla u|^{2}+A\rho(x)^{2}
\end{equation*}
where $\rho(x)$ is defined as \eqref{def-rho}, $\lambda_{max}(x)$ is the biggest eigenvalue of the Hessian matrix $u_{ij}$; $\beta$, $a$ and $A$ are positive constants which will be determined later. 

Suppose that $\tilde{P}(x)$ attains its maximum value in $\Omega$ at $x_{0}$. In the following we distinguish two cases according to the position of $x_{0}$. 

We first consider the case $x_{0} \notin \operatorname{Cut}(p_{0})$ (Note that when the sectional curvature of $M^{n}$ is nonpositive, the cut locus $\operatorname{Cut}(p_{0})=\emptyset$; hence the case $x\notin\operatorname{Cut}(p_{0})$ is automatically satisfied for all $x$. 
For the uniformity of the proof, we keep the same case distinction as in the general setting). 

Choosing a local orthonormal frame field $\{\xi_{i}\}_{i=1}^{n}$ at $x_{0}$ such that $g_{ij}(x_{0})=\delta_{ij}$, and then rotating the coordinate axes, we can diagonalize the Hessian matrix $\nabla^{2}u=(u_{ij})$,
\begin{equation*}
u_{ij}(x_{0})=u_{ii}(x_{0})\delta_{ij},\quad\quad u_{11}(x_{0})\geq u_{22}(x_{0})\geq\cdots\geq u_{nn}(x_{0}),
\end{equation*}
and then
\begin{equation*}
U_{11}(x_{0})\leq U_{22}(x_{0})\leq\cdots\leq U_{nn}(x_{0}).
\end{equation*}
So, by \eqref{prop-7-ieq2} we obtain
\begin{equation}\label{5-ieq1}
F^{11}(x_{0})\geq F^{22}(x_{0})\geq\cdots\geq F^{nn}(x_{0})>0,
\end{equation}
\begin{equation}\label{5-ieq2}
0<T^{11}(x_{0})\leq T^{22}(x_{0})\leq\cdots\leq T^{nn}(x_{0}).
\end{equation}

Now, we define a new function on $\Omega$
\begin{equation*}
P(x)=\beta\log(-u)+\log u_{11}(x)+\frac{a}{2}|\nabla u|^{2}+A\rho(x)^{2},
\end{equation*}
which also attains a local maximum at $x_{0}$. Differentiating $P$ at $x_{0}$ once yields
\begin{equation}\label{5-eq3}
\frac{\beta u_{i}}{u}+\frac{u_{11i}}{u_{11}}+au_{ii}u_{i}+A(\rho^{2})_{i}=0,
\end{equation}
and differentiating $P$ at $x_{0}$ twice,
\begin{equation}\label{5-ieq4}
\frac{\beta u_{ii}}{u}-\frac{\beta u_{i}^{2}}{u^{2}}+\frac{u_{11ii}}{u_{11}}-\frac{u_{11i}^{2}}{u_{11}^{2}}+a\sum_{p=1}^{n}u_{pii}u_{i}+au_{ii}^{2}+A(\rho^{2})_{ii}\leq0.
\end{equation}
Thus, at $x_{0}$
\begin{equation}\label{5-ieq5}
\begin{aligned}
0\geq&~T^{ii}P_{ii}\\
\geq&~\frac{\beta T^{ii}u_{ii}}{u}-\frac{\beta T^{ii}u_{i}^{2}}{u^{2}}+\frac{T^{ii}u_{11ii}}{u_{11}}-\frac{T^{ii}u_{11i}^{2}}{u_{11}^{2}}\\
&+a\sum_{p=1}^{n}T^{ii}u_{pii}u_{i}+aT^{ii}u_{ii}^{2}+AT^{ii}(\rho^{2})_{ii}.
\end{aligned}
\end{equation}
We now begin to estimate each term in \eqref{5-ieq5}. A direct calculation implies
\begin{equation}\label{5-ieq6}
\begin{aligned}
T^{ii}u_{ii}=&\sum_{i=1}^{n}\left(\tau\sum_{j=1}^{n}F^{jj}-F^{ii}\right)\left(\frac{\tau}{n\tau-1}\sum_{p=1}^{n}U_{pp}-U_{ii}\right)\\
=&\sum_{i=1}^{n}F^{ii}U_{ii}+\frac{n\tau^{2}}{n\tau-1}\sum_{j=1}^{n}F^{jj}\sum_{p=1}^{n}U_{pp}-\frac{\tau}{n\tau-1}\sum_{i=1}^{n}F^{ii}\sum_{p=1}^{n}U_{pp}-\tau\sum_{j=1}^{n}F^{jj}\sum_{i=1}^{n}U_{ii}\\
=&\sum_{i=1}^{n}F^{ii}U_{ii}.
\end{aligned}
\end{equation}
From Proposition \ref{prop-2} (5), we have
\begin{equation}\label{5-ieq7}
\begin{aligned}
T^{ii}u_{ii}=&F^{ii}U_{ii}\\
=&\frac{1}{n-l}\left[\frac{S_{n}(U)}{S_{l}(U)}\right]^{\frac{1}{n-l}-1}\frac{(n-l)S_{n}(U)S_{l}(U)-\alpha[\sigma_{l}(U)\sigma_{n-1}(U)-\sigma_{n}(U)\sigma_{l-1}(U)]}{S_{l}^{2}(U)}\\
=&\psi-\alpha Q.
\end{aligned}
\end{equation}
where $\psi=f^{\frac{1}{n-l}}$, $Q=\frac{\psi[\sigma_{l}(U)\sigma_{n-1}(U)-\sigma_{n}(U)\sigma_{l-1}(U)]}{(n-l)fS_{l}^{2}(U)}$. It follows directly from \eqref{N-M-ieq} that $Q>0$. Rewriting equation \eqref{Hessian quo-2} as
\begin{equation}\label{5-eq8}
F(U)=\psi,
\end{equation}
differentiating \eqref{5-eq8} once gives
\begin{equation*}
F^{ii}U_{iip}=\psi_{p},
\end{equation*}
which yields
\begin{equation}\label{5-eq9}
T^{ii}u_{iip}=\psi_{p},
\end{equation}
differentiating \eqref{5-eq8} twice gives
\begin{equation}\label{5-eq10}
F^{ij,rs}U_{iip}U_{rsp}+F^{ii}U_{iipp}=\psi_{pp}.
\end{equation}
By \eqref{3-ieq8} and \eqref{5-eq3}, we estimate $\psi_{11}$ as follows
\begin{equation}\label{5-ieq11}
\begin{aligned}
|\psi_{11}|\leq&~C(1+u_{11}+u_{11}^{2})+|\sum_{i=1}^{n}\partial_{u_{i}}\psi R^{m}_{1i1}u_{m}|+|\sum_{i=1}^{n}\partial_{u_{i}}\psi u_{11i}|\\
\leq&~C+(C+CA)u_{11}+(C+Ca)u_{11}^{2}+\frac{C\beta u_{11}}{-u}
\end{aligned}
\end{equation}
where $C$ is a positive constant depending on $n$, $l$, $\sup|\nabla\rho|$, $||u||_{C^{1}}$, $||R_{ijk}^{l}||_{C^{0}}$ and $||f||_{C^{2}}$.
From \eqref{3-eq7-1} and \eqref{3-eq7-2}, we obtain
\begin{equation}\label{5-eq12-1}
T^{ii}u_{11ii}\geq T^{ii}u_{ii11}-(Cu_{11}+C)\sum_{i=1}^{n}T^{ii},
\end{equation}
\begin{equation}\label{5-eq12-2}
T^{ii}u_{pii}u_{i}\geq T^{ii}u_{iip}u_{i}-C\sum_{i=1}^{n}T^{ii}.
\end{equation}
Then applying the concavity of $F$ together with \eqref{prop-7-ieq1}, \eqref{5-eq10} and \eqref{5-ieq11} yields at $x_{0}$
\begin{equation*}
\begin{aligned}
F^{ii}U_{ii11}&\geq-F^{ij,rs}U_{ii1}U_{rs1}-C-CAu_{11}-(C+Ca)u_{11}^{2}+\frac{C\beta u_{11}}{u}\\
&\geq -2\sum_{i=2}^{n}F^{1i,i1}U_{1i1}^{2}-C-CAu_{11}-(C+Ca)u_{11}^{2}+\frac{C\beta u_{11}}{u}\\
&\geq -2\sum_{i=2}^{n}F^{1i,i1}u_{1i1}^{2}-C-CAu_{11}-(C+Ca)u_{11}^{2}+\frac{C\beta u_{11}}{u},
\end{aligned}
\end{equation*}
which by \eqref{5-eq12-1} can be rewritten as
\begin{equation}\label{5-eq13}
\begin{aligned}
T^{ii}u_{11ii}\geq&~ -2\sum_{i=2}^{n}F^{1i,i1}u_{1i1}^{2}-C-CAu_{11}-(C+Ca)u_{11}^{2}\\
&+\frac{C\beta u_{11}}{u}-(Cu_{11}+C)\sum_{i=1}^{n}T^{ii}.
\end{aligned}
\end{equation}

If $x_{0}=p_{0}$, $\rho(x)$ is not differential at $p_{0}$ at this time; however, a standard fact in Riemannian geometry gives $\nabla^{2}(\rho^{2}) = 2g$ at $p_{0}$. Then we have
\begin{equation*}
(\rho^{2})_{ii}=\nabla^{2}(\rho^{2})(\frac{\partial}{\partial\xi^{i}},\frac{\partial}{\partial\xi^{i}})= 2g(\frac{\partial}{\partial\xi^{i}},\frac{\partial}{\partial\xi^{i}})=2.
\end{equation*}
If $x_{0}\neq p_{0}$, $\rho(x)$ is differentiable at $x_{0}$. Then by Lemma \ref{Hess-compar}, we have on $\Omega\backslash \operatorname{Cut}(p_{0})\cup\{p_{0}\}$ the following estimate
\begin{itemize}
\item If $\overline{K}_{\Omega}\le 0$, then no restriction on the size of $\Omega$ is needed, and the Hessian comparison theorem yields
\begin{equation*}
(\rho^{2})_{ii}=2\rho\rho_{ii}+2\rho_{i}^{2}\ge 2,
\end{equation*}
\item If $\overline{K}_{\Omega}>0$, then $\overline{\Omega}\subset B_{r}(p_{0})$ with $0<r<\frac{\pi}{2\sqrt{K}}$, and hence
\begin{equation*}
(\rho^{2})_{ii}\geq2\rho\sqrt{K}\cot(\sqrt{K}\rho)(1-\rho_{i}^{2})+2\rho_{i}^{2}\geq 2r\sqrt{K}\cot(\sqrt{K}r)>0.
\end{equation*}
\end{itemize}
Therefore, combining the above cases, we obtain
\begin{equation}\label{rho-hessian-ieq}
A(\rho^{2})_{ii}\geq c_{0}A>0,
\end{equation}
where $c_{0}$ is a positive constant depending on $||R_{ijk}^{l}||_{C^{0}}$. Here $A$ is a positive constant to be chosen sufficiently large later. For convenience, we regard $c_{0}A$ as a new constant and still denote it by $A$ in the following.
Plugging \eqref{5-ieq7}, \eqref{5-eq9}, \eqref{5-eq12-2}, \eqref{5-eq13} and \eqref{rho-hessian-ieq} into \eqref{5-ieq5}, assuming $u_{11}\geq1$, we have at $x_{0}$
\begin{equation}\label{5-ieq14}
\begin{aligned}
0\geq&~T^{ii}P_{ii}\\
\geq&~\frac{C\beta}{u}-\frac{\alpha\beta Q}{u}-\frac{\beta T^{ii}u_{i}^{2}}{u^{2}}-\frac{2}{u_{11}}\sum_{i=2}^{n}F^{1i,i1}u_{1i1}^{2}-C(1+A)-(C+Ca)u_{11}\\
&-C(1+a)\sum_{i=1}^{n}T^{ii}-\frac{T^{ii}u_{11i}^{2}}{u_{11}^{2}}+aT^{ii}u_{ii}^{2}+A\sum_{i=1}^{n}T^{ii}.
\end{aligned}
\end{equation}

Now we should divide our proof into two cases to deal with the bad term $-C u_{11}$ and third derivative terms, where the positive constant $\delta$ will be determined later. Note that in the estimates that follow, the letter $C$ represents a generic constant depending only on the known data of the problem (i.e. $n$, $l$, $\tau$, $\sup|\nabla\rho|$, $||u||_{C^{1}}$, $||R_{ijk}^{l}||_{C^{1}}$, $||f||_{C^{2}}$), and its value may change from line to line.

\textbf{Case 1.} $|u_{ii}|(x_{0})\leq\delta u_{11}(x_{0})$ for all $i\geq 2$.

To handle the bad term $-Cu_{11}$, we first establish the following lemma.

\begin{lemma}\label{5-lemma-1}
If we choose $u_{11}(x_{0})$ and $A$ sufficiently large, and $\delta$ sufficiently small, then we obtain at $x_{0}$
\begin{equation}\label{5-lemma-1eq1}
\sum_{i=1}^{n}T^{ii}+Q\geq Cu_{11}.
\end{equation}
\end{lemma}
\begin{proof}
All the following calculation are done at $x_{0}$. It follows consequently
\begin{equation*}
-(n-1)\delta u_{11}\leq U_{11}\leq[(\tau-1)+(n-1)\tau\delta]u_{11},
\end{equation*}
and
\begin{equation*}
[1-(n-2)\delta]u_{11}\leq U_{22}\leq\cdots\leq U_{nn}\leq[\tau-\delta+(n-1)\tau\delta]u_{11}.
\end{equation*}
Then, choosing $\delta$ sufficiently small and using the inequalities above, we obtain for $k\leq n-1$
\begin{equation*}
\begin{aligned}
\sigma_{k}(\eta)=&~\sigma_{k}(\eta|1)+\eta_{1}\sigma_{k-1}(\eta|1)\\
\geq&~C_{n-1}^{k}[1-(n-2)\delta]^{k}u_{11}^{k}-C_{n-1}^{k-1}(n-1)\delta[\tau-\delta+(n-1)\tau\delta]^{k-1}u_{11}^{k}\\
\geq&~\frac{u_{11}^{k}}{2}
\end{aligned}
\end{equation*}
where $\eta=(U_{11}, U_{11},\cdots, U_{nn})$ and we get the first equality by Proposition \ref{prop-1} (3). Similarly,
\begin{equation*}
\begin{aligned}
\sigma_{k}(\eta)=&~\sigma_{k}(\eta|1)+\eta_{1}\sigma_{k-1}(\eta|1)\\
\leq&~C_{n-1}^{k}[1-(n-2)\delta]^{k}u_{11}^{k}+C_{n-1}^{k-1}[(\tau-1)+(n-1)\tau\delta][\tau-\delta+(n-1)\tau\delta]^{k-1}u_{11}^{k}\\
\leq&~C(n,k)\tau^{k}u_{11}^{k}.
\end{aligned}
\end{equation*}
Moreover, by Proposition \ref{prop-2} (4) and \eqref{prop-5-ieq1}, we have
\begin{equation*}
\begin{aligned}
\sum_{i=1}^{n}T^{ii}+Q=&~(n\tau-1)\sum_{i=1}^{n}F^{ii}+Q\\
\geq&~\frac{1}{n-l}f^{\frac{1}{n-l}-1}\frac{2S_{n-1}S_{l}-(n-l+2)S_{l-1}S_{n}-\sigma_{n-1}\sigma_{l}+\sigma_{n}\sigma_{l-1}}{S_{l}^{2}}+Q\\
=&~\frac{1}{n-l}f^{\frac{1}{n-l}-1}\frac{2S_{n-1}S_{l}-(n-l+2)S_{l-1}S_{n}}{S_{l}^{2}}\\
\geq&~\frac{2}{n}f^{\frac{1}{n-l}-1}\frac{S_{n-1}}{S_{l}}.
\end{aligned}
\end{equation*}
Choosing $u_{11}$ sufficiently large and $n\geq l+2$,
\begin{equation*}
\sum_{i=1}^{n}T^{ii}+Q\geq C\tau^{-l}u_{11}^{n-1-l}\geq Cu_{11}.
\end{equation*}
So, we complete the proof.
\end{proof}

Next, we deal with the third derivatives terms.

\begin{lemma}\label{5-lemma-2}
We have at $x_{0}$ for $\delta\leq1$,
\begin{equation}\label{5-lemma-2eq1}
\begin{aligned}
\frac{2}{1+\delta}\sum_{i=2}^{n}\frac{T^{ii}u_{11i}^{2}}{u_{11}^{2}}\leq& -\frac{2}{u_{11}}\sum_{i=2}^{n}F^{1i,i1}u_{11i}^{2}+C\frac{\beta^{2}T^{11}}{u^{2}}+Ca^{2}\delta^{2}T^{11}u_{11}^{2}+CA^{2}T^{11}.\\
\end{aligned}
\end{equation}
\end{lemma}
\begin{proof}
All the following calculation are done at $x_{0}$. Since $|u_{ii}|\leq\delta u_{11}$ for $i\geq2$, we have for all $i\geq 2$
\begin{equation*}
\frac{1}{u_{11}}\leq\frac{1+\delta}{u_{11}-u_{ii}},
\end{equation*}
then, we know from \eqref{prop-7-ieq1}
\begin{equation*}
\begin{aligned}
-\frac{2}{u_{11}}\sum_{i=2}^{n}F^{1i,i1}u_{11i}^{2}\geq&-\frac{2}{u_{11}}\sum_{i=2}^{n}\frac{F^{11}-F^{ii}}{U_{11}-U_{ii}}u_{11i}^{2}\\
=&-\frac{2}{u_{11}}\sum_{i=2}^{n}\frac{T^{ii}-T^{11}}{u_{ii}-u_{11}}u_{11i}^{2}\\
\geq&~\frac{2}{1+\delta}\sum_{i=2}^{n}\frac{T^{ii}-T^{11}}{u_{11}^{2}}u_{11i}^{2}.
\end{aligned}
\end{equation*}
Thus, we obtain
\begin{equation}\label{5-lemma-2ieq2}
\frac{2}{1+\delta}\sum_{i=2}^{n}\frac{T^{ii}u_{11i}^{2}}{u_{11}^{2}}\leq-\frac{2}{u_{11}}\sum_{i=2}^{n}F^{1i,i1}u_{11i}^{2}
+\frac{2}{1+\delta}\sum_{i=2}^{n}\frac{T^{11}u_{11i}^{2}}{u_{11}^{2}}.
\end{equation}
Using Cauchy-Schwarz inequality, from \eqref{5-eq3} at $x_{0}$ we get
\begin{equation}\label{5-lemma-2ieq3}
\begin{aligned}
\sum_{i=2}^{n}\frac{T^{11}u_{11i}^{2}}{u_{11}^{2}}\leq&~3\sum_{i=2}^{n}\frac{\beta^{2}T^{11}u_{i}^{2}}{u^{2}}+3a^{2}\sum_{i=2}^{n}T^{11}u_{i}^{2}u_{ii}^{2}+3A^{2}\sum_{i=2}^{n}T^{11}(\rho^{2})_{i}^{2}\\
\leq&~C\frac{\beta^{2}T^{11}}{u^{2}}+Ca^{2}\delta^{2}T^{11}u_{11}^{2}+CA^{2}T^{11}.
\end{aligned}
\end{equation}
Combining \eqref{5-lemma-2ieq2} and \eqref{5-lemma-2ieq3}, we complete the proof.
\end{proof}

Applying Young's inequality to equation \eqref{3-eq7-1}, we obtain
\begin{equation*}
u_{1i1}^{2}=(u_{11i}+u_{j}R^{j}_{1i1})^{2}\geq(1-\gamma)u_{11i}^{2}-(\frac{1}{\gamma}-1)(u_{j}R^{j}_{1i1})^{2},
\end{equation*}
where $0 < \gamma < 1$ is a positive constant to be chosen later. From the above inequality and \eqref{prop-7-ieq1}, we have
\begin{equation}\label{5-ieq15}
\begin{aligned}
-2\sum_{i=2}^{n}F^{1i,i1}u_{1i1}^{2}\geq&~-2(1-\gamma)\sum_{i=2}^{n}F^{1i,i1}u_{11i}^{2}+C(\frac{1}{\gamma}-1)\sum_{i=2}^{n}\frac{T^{11}-T^{ii}}{u_{11}-u_{ii}}\\
\geq&~-2(1-\gamma)\sum_{i=2}^{n}F^{1i,i1}u_{11i}^{2}-\frac{C}{(1-\delta)u_{11}}(\frac{1}{\gamma}-1)\sum_{i=2}^{n}T^{ii}
\end{aligned}
\end{equation}

Choose $\gamma$ and $a$ sufficiently small, and $A$ sufficiently large (with $A$ depending on $\tau^{-l}$). Substituting \eqref{5-lemma-1eq1}, \eqref{5-lemma-2eq1} and \eqref{5-ieq15} into \eqref{5-ieq14} yields
\begin{equation}\label{5-ieq16}
\begin{aligned}
0\geq&~T^{ii}P_{ii}\\
\geq&~\frac{C\beta}{u}-\frac{\alpha\beta Q}{u}-\frac{\beta T^{ii}u_{i}^{2}}{u^{2}}+\left(\frac{2(1-\gamma)}{1+\delta}-1\right)\sum_{i=2}^{n}\frac{T^{ii}u_{11i}^{2}}{u_{11}^{2}}\\
&-\frac{T^{11}u_{111}^{2}}{u_{11}^{2}}-C\frac{\beta^{2}T^{11}}{u^{2}}-Ca^{2}\delta^{2}T^{11}u_{11}^{2}-CA^{2}T^{11}\\
&+aT^{ii}u_{ii}^{2}+\frac{A}{2}\sum_{i=1}^{n}T^{ii}+Cu_{11}-\frac{A}{2}Q.
\end{aligned}
\end{equation}
By Cauchy-Schwarz inequality, from \eqref{5-eq3} and \eqref{5-ieq16} we obtain at $x_{0}$
\begin{equation*}
\begin{aligned}
0\geq&~T^{ii}P_{ii}\\
\geq&~\left(\frac{2(1-\gamma)}{1+\delta}-1-\frac{3}{\beta}\right)\sum_{i=2}^{n}\frac{T^{ii}u_{11i}^{2}}{u_{11}^{2}}+(\frac{a}{2}-\frac{Ca^{2}}{\beta})\sum_{i=2}^{n}T^{ii}u_{ii}^{2}\\
&+(\frac{a}{2}-Ca^{2}-Ca^{2}\delta^{2})T^{11}u_{11}^{2}+(\frac{A}{2}-\frac{CA^{2}}{\beta})\sum_{i=2}^{n}T^{ii}\\
&+\frac{A}{2}T^{11}-CA^{2}T^{11}-C\frac{\beta^{2}T^{11}}{u^{2}}+Cu_{11}+\frac{C\beta}{u}-\frac{\alpha\beta Q}{u}-\frac{A}{2}Q.
\end{aligned}
\end{equation*}
Fixing the chosen $\delta$, $\gamma$, $A$ as before. Now we further choose $a$ sufficiently small (possibly even smaller than before), and choosing $\beta$ sufficiently large. Then at $x_{0}$ we obtain
\begin{equation*}
0\geq\frac{a}{4}T^{11}u_{11}^{2}-(CA^{2}+C\frac{\beta^{2}}{u^{2}})T^{11}+\frac{C}{u}+Cu_{11}.
\end{equation*}
Choosing $(-u)^{2}u_{11}^{2}(x_{0})$ sufficiently large, we have at $x_{0}$
\begin{equation*}
0\geq\frac{C}{u}+Cu_{11},
\end{equation*}
which implies at $x_{0}$
\begin{equation*}
(-u)^{\beta}u_{11}\leq C.
\end{equation*}
So we complete the proof for case 1.

\textbf{Case 2.} $u_{22}(x_{0})>\delta u_{11}(x_{0})$ or $u_{nn}(x_{0})<-\delta u_{11}(x_{0})$ at $x_{0}$.

First, we need the following lemma to absorb the bad term $-Cu_{11}$.

\begin{lemma}\label{5-lemma-3}
If we choose $u_{11}(x_{0})\geq\frac{\gamma}{a\delta^{2}}$, then we have at $x_{0}$,
\begin{equation}\label{5-lemma-3ieq1}
\frac{a}{2}\sum_{i=1}^{n}T^{ii}u_{ii}^{2}\geq\gamma Cu_{11}.
\end{equation}
\end{lemma}
\begin{proof}
All the following calculation are done at $x_{0}$. It follows from $u_{22}(x_{0})>\delta u_{11}(x_{0})$ or $u_{nn}(x_{0})<-\delta u_{11}(x_{0})$ at $x_{0}$,
\begin{equation}\label{5-lemma-3ieq2}
\sum_{i=1}^{n}T^{ii}u_{ii}^{2}\geq T^{22}u_{22}^{2}\geq\delta^{2}T^{22}u_{11}^{2}.
\end{equation}
Furthermore, we have by Lemma \ref{prop-6} (6)
\begin{equation}\label{5-lemma-3ieq3}
T^{22}=\tau\sum_{i=1}^{n}F^{ii}-F^{22}\geq F^{11}\geq\frac{1}{n}\sum_{i=1}^{n}F^{ii}\geq C.
\end{equation}
Plugging inequality \eqref{5-lemma-3ieq3} into \eqref{5-lemma-3ieq2} gives
\begin{equation}\label{5-lemma-3ieq4}
\frac{a}{2}\sum_{i=1}^{n}T^{ii}u_{ii}^{2}\geq Ca\delta^{2}u_{11}^{2}\frac{1}{\tau}\sum_{i=1}^{n}T^{ii}.
\end{equation}
and
\begin{equation*}
\frac{a}{2}\sum_{i=1}^{n}T^{ii}u_{ii}^{2}\geq Ca\delta^{2}u_{11}^{2}.
\end{equation*}
Choosing $u_{11}\geq\frac{\gamma}{a\delta^{2}}$ yields
\begin{equation*}
\frac{a}{2}\sum_{i=1}^{n}T^{ii}u_{ii}^{2}\geq \gamma Cu_{11}.
\end{equation*}
\end{proof}

By \eqref{5-lemma-3ieq1}, we choose $\gamma$ and $A$ sufficiently large. Then from \eqref{5-ieq14} we obtain that
\begin{equation}\label{5-ieq19}
\begin{aligned}
0\geq&~T^{ii}P_{ii}\\
\geq&~\frac{C\beta}{u}-\frac{\beta T^{ii}u_{i}^{2}}{u^{2}}-\frac{2}{u_{11}}\sum_{i=2}^{n}F^{1i,i1}u_{1i1}^{2}-\frac{T^{ii}u_{11i}^{2}}{u_{11}^{2}}\\
&+\frac{a}{2}T^{ii}u_{ii}^{2}+\frac{A}{2}\sum_{i=1}^{n}T^{ii}+Cu_{11}.
\end{aligned}
\end{equation}
By Cauchy-Schwarz inequality, from \eqref{5-eq3}, \eqref{5-lemma-3ieq4} and \eqref{5-ieq19}, fix the chosen $\beta$, $A$, $\gamma$ and $\delta$ as before and choosing sufficiently small $a$ yields at $x_{0}$
\begin{equation}\label{5-ieq21}
\begin{aligned}
0\geq&~T^{ii}P_{ii}\\
\geq&~\frac{C\beta}{u}-(\beta+C\beta^{2})\frac{ T^{ii}u_{i}^{2}}{u^{2}}+(\frac{a}{2}-Ca^{2})T^{ii}u_{ii}^{2}\\
&+(\frac{A}{2}-CA^{2})\sum_{i=1}^{n}T^{ii}+Cu_{11}\\
\geq&~\sum_{i=1}^{n}T^{ii}(Cu_{11}^{2}-\frac{C}{u^{2}}-C)+\frac{C}{u}+Cu_{11}\\
\geq&~\frac{C}{u}+Cu_{11},
\end{aligned}
\end{equation}
where we discard the positive term $-\frac{2}{u_{11}}\sum_{i=2}^{n}F^{1i,i1}u_{11i}^{2}$ to get the second inequality and we choose $(-u)^{2}u_{11}^{2}(x_{0})$ sufficiently large to get the last inequality. 

Next, we consider the case $x_{0} \in \operatorname{Cut}(p_{0})$. By the similar method as in the proof of Theorem \ref{maintheorem1}, we consider the test function
\begin{equation*}
\bar{P}(x)=\beta\log(-u)+\log \lambda_{max}(x)+\frac{a}{2}|\nabla u|^{2}+\frac{A}{2}(\rho_{1}(x)+\varepsilon)^{2}
\end{equation*}
at a neighbourhood of $x_{0}$. Then we obtain
\begin{equation*}
\bar{P}(x)\leq \tilde{P}(x)\leq \tilde{P}(x_{0})\leq\bar{P}(x_{0}),
\end{equation*}
which implies that $x_{0}$ is still a maximum point of $\bar{P}$. Then, by the same calculations and application of the Hessian comparison theorem as in the case $x_0 \notin \operatorname{Cut}(p_{0})$, we can similarly obtain $(-u)^{\beta}u_{11}\leq C$. Thus, we complete the proof.

\end{document}